\documentclass[psamsfonts]{amsart}

\usepackage{amssymb,amsfonts}
\usepackage[all,arc]{xy}
\usepackage{enumerate}
\usepackage{mathrsfs}
\usepackage{hyperref}
\usepackage{graphicx}
\usepackage{mathtools}
\usepackage[dvipsnames]{xcolor}

\newtheorem{thm}{Theorem}[section]
\newtheorem{cor}[thm]{Corollary}
\newtheorem{prop}[thm]{Proposition}
\newtheorem{lem}[thm]{Lemma}

\theoremstyle{definition}
\newtheorem{defn}[thm]{Definition}

\newtheorem{exmp}[thm]{Example}

\theoremstyle{remark}
\newtheorem{rem}[thm]{Remark}

\makeatletter
\let\c@equation\c@thm
\makeatother
\numberwithin{equation}{section}

\usepackage{graphicx}

\title{Free Semigroups of Large Critical Exponent}
\author{Aleksander Skenderi}
\address{Department of Mathematics, University of Wisconsin-Madison}
\email{askenderi@wisc.edu}
\begin{document}
\begin{abstract}
For a convergence group equipped with an expanding coarse-cocycle, we construct finitely generated free subsemigroups, which we call $\textit{Bishop--Jones}$ $\textit{semigroups}$, of critical exponent arbitrarily close to but strictly less than the critical exponent of the ambient group. As an application, 
we show that for any non-elementary transverse subgroup $\Gamma$ of a semisimple Lie group $G$, there exist finitely generated free Anosov subsemigroups in the sense of Kassel--Potrie of critical exponent arbitrarily close to but strictly less than that of the ambient transverse group. Furthermore, we show that these semigroups admit quasi-isometric embeddings into the symmetric space $X$ of $G$ with certain additional coarse-geometric properties.
\end{abstract}
\maketitle
\section{Introduction}
An important object in the study of discrete subgroups of Lie groups and discrete subgroups of isometries of Gromov hyperbolic spaces is the $\textit{critical exponent}$. For concreteness, let $\Gamma \subset \mathrm{SO}^{\circ}(n,1) \cong \mathrm{Isom}^{+}(\mathbf{H}_{\mathbb{R}}^{n})$ be a discrete subgroup of isometries of the real hyperbolic $n$-space and fix a basepoint $x \in \mathbf{H}_{\mathbb{R}}^{n}$. The $\textit{critical exponent}$ of $\Gamma$ is the abscissa of convergence of the $\textit{Poincar\'e series}$
\begin{align*}
Q(s) = \sum_{\gamma \in \Gamma} e^{-s d_{\mathbf{H}_{\mathbb{R}}^{n}}(x,\gamma x)},
\end{align*}
i.e.,
\begin{align*}
\delta (\Gamma) = \inf \{ s > 0 : Q(s) < \infty \} \in [0, \infty].
\end{align*} Equivalently, the critical exponent is given by 
\begin{align*}
    \delta (\Gamma) = \limsup_{T \to \infty} \frac{1}{T} \log \# \{\gamma \in \Gamma : d_{\mathbf{H}_{\mathbb{R}}^{n}}(x,\gamma x) \leq T \},
\end{align*} and thus it measures the asymptotic growth rate of the $\Gamma$-orbits in $\mathbf{H}_{\mathbb{R}}^{n}$. We remark that the critical exponent is independent of the choice of basepoint $x \in \mathbf{H}_{\mathbb{R}}^{n}$, hence is well-defined. \par 
A classical theorem of Bishop--Jones in \cite{BJ} shows that for any non-elementary Kleinian group $\Gamma \subset \mathrm{SO}^{\circ}(3,1)$, the Hausdorff dimension of its conical limit set is equal to its critical exponent. Showing that the Hausdorff dimension is bounded above by the critical exponent is well-known and not particularly difficult. The great novelty in their proof was to construct a well-behaved tree in $\mathbf{H}_{\mathbb{R}}^{3}$, related to the group $\Gamma$, and whose set of accumulation points lands inside of the conical limit set of $\Gamma$. One uses properties of the tree to construct a probability measure supported on its set of accumulation points, and then uses Frostman's Lemma to obtain a lower bound on the Hausdorff dimension.
\par 
While this is a truly ingenious construction, the tree is not so easy to work with. Indeed, their construction is not $\Gamma$-equivariant, and hence does not play particularly well with the $\Gamma$-action on $\mathbf{H}_{\mathbb{R}}^{3}$. Motivated by the work of Bishop--Jones, we construct finitely generated free subsemigroups in a broad class of convergence groups. These semigroups, which we call $\textit{Bishop}$--$\textit{Jones semigroups}$, are large from the perspective of the critical exponent of the ambient group. Nevertheless, we do not prove an analog of Bishop--Jones' theorem, since in the general abstract setting we work in, there is no (at least to the author) natural way to measure the Hausdorff dimensions of (conical) limit sets.
\par 
To better illustrate our result, we describe two applications of our main result Theorem \ref{MainThm}. The first application of Theorem \ref{MainThm} pertains to a class of discrete subgroups of semisimple Lie groups known as $\textit{transverse groups}$. Transverse groups (called regular antipodal groups by Kapovich--Leeb--Porti \cite{KLP}) have been widely studied in recent years by many authors; see for instance \cite{CZZ1}, \cite{CZZ2}, \cite{KLP}, and \cite{KOW}. The class of transverse groups include both Anosov and relatively Anosov groups, as well as all discrete subgroups of rank one Lie groups. Given a non-elementary transverse subgroup $\Gamma$ of a connected semisimple real Lie group $G$ without compact factors and with finite center, we show that its critical exponent (with respect to an appropriate linear functional on the partial Cartan subspace) can be approximated by the critical exponents of finitely generated free subsemigroups, which are Anosov in the sense of Kassel--Potrie \cite{KP}. Delaying definitions until section \ref{BackgroundSection}, we prove the following.
\begin{thm} [Theorem \ref{TransverseCritExp}]
\label{TransverseCritExpIntro}    
Suppose $\Gamma$ is a non-elementary $P_{\theta}$-transverse group and $\phi \in \mathfrak{a}_{\theta}^{*}$ is such that $\phi(\kappa(\gamma_{n})) \rightarrow \infty$ for any sequence $\{\gamma_{n}\}$ of pairwise distinct elements of $\Gamma$. Then there exists a sequence $\{\Gamma_{n}\}_{n \geq 1}$ of free $P_{\theta}$-Anosov subsemigroups of $\Gamma$ so that 
\begin{itemize}
    \item[(1)] $\delta^{\phi}(\Gamma_{n}) < \delta^{\phi}(\Gamma)$ for all $n \geq 1$, and
    \item[(2)] $\lim_{n \to \infty} \delta^{\phi}(\Gamma_{n}) = \delta^{\phi}(\Gamma)$.
\end{itemize} 
\end{thm}
We also show that the orbit of a Bishop--Jones semigroup embeds into the symmetric space $X = G/K$ as a $(\mathcal{C},B)$-regular map in the sense of Kapovich--Leeb--Porti \cite{KLP} by using work of Dey--Kim--Oh \cite{DKO}. The precise statement of this result is Theorem \ref{CoarseTriangleInequality}.
\par
The second application is related to the classical critical exponent gap theorem of Corlette for discrete subgroups in Sp$(n,1)$ and $\mathrm{F}_{4}^{-20}$. Denote by $\mathbf{H}_{\mathbb{H}}^{n}$ the $n$-dimensional quaternionic hyperbolic space and $\mathbf{H}_{\mathbb{O}}^{2}$ the Cayley hyperbolic plane (also known as the octonionic projective plane). These are connected, contractible, negatively curved Riemannian manifolds with normalized sectional curvatures between $-4$ and $-1$. The simple Lie groups of real rank one Sp$(n,1)$ and $\mathrm{F}_{4}^{-20}$ are the orientation-preserving isometry groups of $\mathbf{H}_{\mathbb{H}}^{n}$ and $\mathbf{H}_{\mathbb{O}}^{2}$, respectively. For a discrete subgroup $\Gamma$ in either of these isometry groups, the critical exponent $\delta(\Gamma)$ is defined analogously as in the real hyperbolic case. Corlette's renowned gap theorem states the following. \newline
\begin{thm} [Corlette, Theorem 4.4 of \cite{Cor}] ~\
\label{CorletteGapThm}
\begin{itemize}
    \item[(1)] If $\Gamma \subset \mathrm{Sp}(n,1)$, $n \geq 2$, is a discrete subgroup, then $\delta(\Gamma) = 4n+2$ or $\delta(\Gamma) \leq 4n$. Moreover, $\delta(\Gamma) = 4n+2$ if and only if $\Gamma$ is a lattice.
    \item[(2)] If $\Gamma \subset \mathrm{F}_{4}^{-20}$ is a discrete subgroup, then $\delta(\Gamma) = 22$ or $\delta(\Gamma) \leq 16$. Moreover, $\delta(\Gamma) = 22$ if and only if $\Gamma$ is a lattice. 
\end{itemize}
\end{thm}
As a consequence of our main result, Theorem \ref{MainThm}, we show that there is no such gap phenomenon for discrete $\textit{subsemigroups}$ in Sp$(n,1)$ and $\mathrm{F}_{4}^{-20}$. More precisely, we prove the following. \\
\begin{thm} [Theorem \ref{NoGapRank1'}] ~\
\label{NoGapRank1}
\begin{itemize}
    \item[(1)] Let $\Gamma \subset \mathrm{Sp}(n,1)$, $n \geq 2$, be a lattice. Then there exists a sequence $\{\Gamma_{m}\}_{m \geq 1}$ of finitely generated free subsemigroups of $\Gamma$ so that $\delta(\Gamma_{m}) < \delta(\Gamma) = 4n+2$ for all $m \geq 1$ and $\lim_{m \to \infty} \delta(\Gamma_{m}) = 4n+2$.
    \item[(2)] Let $\Gamma \subset \mathrm{F}_{4}^{-20}$ be a lattice. Then there exists a sequence $\{\Gamma_{m}\}_{m \geq 1}$ of finitely generated free subsemigroups of $\Gamma$ so that $\delta(\Gamma_{m}) < \delta(\Gamma) = 22$ for all $m \geq 1$ and $\lim_{m \to \infty} \delta(\Gamma_{m}) = 22$.
\end{itemize}
\end{thm}
We mention an interesting group-theoretic consequence of Theorem \ref{NoGapRank1}. Consider the case when $\Gamma \subset \mathrm{Sp}(n,1)$, $n \geq 2$, is a lattice (an entirely analogous discussion holds when $\Gamma \subset \mathrm{F}_{4}^{-20}$ is a lattice). Then the above theorem tells us that we can find a finitely generated free subsemigroup $\mathcal{T}$, with finite generating set $S$, so that $4n < \delta(\mathcal{T}) < 4n+2$. Let $\Gamma' := \langle S \rangle$ be the group generated by $S$. Then $\delta(\Gamma') \geq \delta (\mathcal{T}) > 4n$, so Corlette's gap theorem implies that $\delta(\Gamma') = 4n+2$ and $\Gamma'$ is a lattice with finite index in $\Gamma$. In particular, there is some relation among the elements of $S$ when one is allowed to multiply by negative powers of these elements, even though there is no relation when multiplying together exclusively positive powers of these elements. 
\begin{rem}
We remark that the analog of Corlette's gap theorem is known to fail for discrete subgroups of $\mathrm{SO}(n,1) \cong \mathrm{Isom}(\mathbf{H}_{\mathbb{R}}^{n})$ (see Sullivan \cite{Sul1} for examples when $n=3$ and Kapovich--Kontorovich \cite{KK} for examples in all dimensions) and for discrete subgroups of $\mathrm{SU}(m,1) \cong \mathrm{Isom}(\mathbf{H}_{\mathbb{C}}^{m})$ when $m = 2,3$ by work of Dey--Liu \cite{DL}.
\end{rem}
\subsection*{Similarities and differences in the work of Yang}
After I posted a first version of this paper to the arXiv, Wenyuan Yang informed me that Theorem \ref{NoGapRank1} also follows from his earlier work ``Statistically convex-cocompact actions of groups with contracting elements'' \cite{Y}. Indeed, Theorem A of Yang's paper shows that if a group $G$ admits a proper action on a geodesic metric space with a contracting element, then there exists a sequence of free subsemigroups of $G$ with critical exponent arbitrarily close to, but strictly less than, that of $G$. Thus his Theorem A also provides a proof of Theorem \ref{NoGapRank1}. Nevertheless, our general framework used to prove our main result Theorem \ref{MainThm} is different from Yang's and thus we obtain a new proof of Theorem \ref{NoGapRank1}. I will now discuss in more detail the differences between Yang's framework and the framework of this paper of expanding coarse-cocyles. 
\par 
Both our results apply to many notable classes of groups, including hyperbolic and relatively hyperbolic groups. Moreover, Yang's work applies for instance to groups acting properly and cocompactly on a CAT(0) space with rank-one elements, whereas the present work applies to groups acting properly on CAT(0) visibility spaces (that is, the action need not be cocompact; see section 1.2.6 of \cite{BCZZ1} for further details). Yang's work also applies to certain small cancellation groups as well as to mapping class groups of closed hyperbolic surfaces. One notable difference in the results of our papers is that Yang's general framework requires the group to act on a properly and by isometries on a \emph{geodesic} metric space, whereas in our abstract setting the groups we consider come equipped with a convergence action on a compact space with a coarse-cocycle.
\par 
A notable difference in applications is that our setting applies to transverse subgroups of higher rank semisimple Lie groups, which need not be hyperbolic nor even relatively hyperbolic (see \cite{CZZ2} for an intimate relation between transverse groups and so-called projectively visible subgroups of the automorphism group of a properly convex domain in real projective space). Due to the presence of flats in higher-rank symmetric spaces, it seems unlikely that one could build a suitable metric for which a transverse group acts by isometries on the symmetric space with a contracting element.

\subsection*{Outline of the Paper}
In section \ref{Prelims}, we recall the relevant concepts from the work \cite{BCZZ1} of Blayac--Canary--Zhu--Zimmer, which is indispensable to this paper. Section \ref{ConstructionSection} is the heart of the paper, where we prove Theorem \ref{MainThm} and develop numerous useful properties of Bishop--Jones semigroups. In section \ref{BackgroundSection}, we recall the relevant notions from Lie theory, transverse groups, and Anosov semigroups for our applications. In section \ref{Applications}, we prove Theorems \ref{TransverseCritExpIntro} and \ref{NoGapRank1}. Lastly, we discuss the relevance of Bishop--Jones semigroups to studying the coarse geometry of the symmetric space in section \ref{CoarseGeom}.
\subsection*{Acknowledgements}
I would like to thank my advisors Andrew Zimmer and Sebastian Hurtado-Salazar for all their help. I am very grateful to Andrew for suggesting the original question that led to this work, for his generosity with his ideas and time during countless discussions, and for his comments and corrections on several drafts of this work. I am very grateful to Sebastian for many helpful discussions and for his interest in my work. I am indebted to both Andrew and Sebastian for their kindness and support, particularly during a difficult time in my personal life earlier in my graduate studies. \par 
I would also like to thank Feng Zhu and Fernando Al Assal for several nice conversations related to this work, Dongryul Kim for helpful discussions regarding his joint work \cite{DKO} with Subhadip Dey and Hee Oh, Subhadip Dey for helpful comments on an earlier version of this paper, and Wenyuan Yang for informing me of his work \cite{Y}. Finally, I thank the anonymous referees for their thorough reading and helpful comments on an earlier version of this paper.
\par 
This material is based upon work supported by the National Science Foundation under Grant No. DMS-2037851 and Grant No. DMS-2230900. 
\section{Preliminaries}
\label{Prelims}
\subsection{Convergence Groups and Expanding Coarse-Cocycles}
When $M$ is a compact metrizable space, a countable subgroup $\Gamma \subset \mathrm{Homeo}(M)$ is called a (discrete) $\textit{convergence group}$ if for every sequence of pairwise distinct elements $\{\gamma_{n}\} \subset \Gamma$, there exist points $a,b \in M$ and a subsequence $\{\gamma_{n_{k}}\}$ so that $\gamma_{n_{k}} \big|_{M \smallsetminus \{b\}}$ converges locally uniformly to $a$ as $k \to \infty$. Blayac--Canary--Zhu--Zimmer recently developed a theory of Patterson--Sullivan measures associated to expanding coarse cocycles of convergence groups and used it to establish counting, mixing, and equidistribution results for a broad class of groups, which, in particular, includes relatively Anosov groups; see \cite{BCZZ1} and \cite{BCZZ2}. In this section, we describe the relevant concepts and results from their work that we will use later in this paper.
\subsubsection{The Basics of Convergence Groups} Given a convergence group $\Gamma \subset \mathrm{Homeo}(M)$, we define the following:
\begin{itemize}
    \item[(1)] The $\textit{limit set}$ $\Lambda(\Gamma)$ is the set of points $x \in M$ for which there exist $y \in M$ and a sequence $\{\gamma_{n}\}$ in $\Gamma$ such that $\gamma_{n}\big|_{M \smallsetminus \{y\}}$ converges locally uniformly to $x$.
    \item[(2)] A point $x \in \Lambda(\Gamma)$ is called a $\textit{conical limit point}$ if there exist distinct points $a,b \in M$ and a sequence $\{\gamma_{n}\}$ in $\Gamma$ so that $\lim_{n \to \infty} \gamma_{n}x = a$ and $\lim_{n \to \infty} \gamma_{n}y = b$ for all $y \in M \smallsetminus \{x\}$. 
\end{itemize} A convergence group $\Gamma$ is said to be $\textit{non-elementary}$ if its limit set $\Lambda(\Gamma)$ contains at least three points. In this case, $\Lambda(\Gamma)$ is the smallest $\Gamma$-invariant closed subset of $M$ (see Theorem 2S of \cite{T}). For the remainder of this paper, we will only consider non-elementary convergence groups.
\par 
In Theorem 2B of \cite{T}, Tukia shows that the elements in a convergence group $\Gamma \subset \mathrm{Homeo}(M)$ can be classified as follows: every element $\gamma \in \Gamma$ is either
\begin{itemize}
    \item[(1)] $\textit{loxodromic}$, meaning that it has two fixed points $\gamma^{+}$ and $\gamma^{-}$ in the limit set $\Lambda(\Gamma) \subset M$ and $(\gamma^{\pm 1})^{n} \big|_{M \smallsetminus \{\gamma^{\mp}\}}$ converges locally uniformly to $\gamma^{\pm}$,
    \item[(2)] $\textit{parabolic}$, meaning that it has one fixed point $p \in \Lambda(\Gamma)$ and $(\gamma^{\pm 1})^{n} \big|_{M \smallsetminus \{p\}}$ converges locally uniformly to $p$, or
    \item[(3)] $\textit{elliptic}$, meaning it has finite order. 
\end{itemize}
In \cite{BCZZ1} it was observed that the set $\Gamma \sqcup M$ has a unique topology which makes it a compact metrizable space on which the natural action of $\Gamma$ is a convergence group action. Any metric $d$ that induces this topology is called a $\textit{compatible metric}$. Given any $x \in \Gamma \sqcup M$, we write $B_{\epsilon}(x)$ for the open ball in $\Gamma \sqcup M$ centered at $x$ of radius $\epsilon$ with respect to the metric $d$. The following result, phrased more generally than in \cite{BCZZ1}, in fact follows from their original proof. 
\begin{lem} [Part (3) of Proposition 2.3 of \cite{BCZZ1}]
\label{TreeLem0}
Let $\Gamma \subset \mathrm{Homeo}(M)$ be a convergence group. For any compatible metric $d$ and any $\epsilon > 0$, there exists a finite set $F \subset \Gamma$ such that 
\begin{align*}
    \gamma \big( (\Gamma \sqcup M) \smallsetminus B_{\epsilon}(\gamma^{-1}) \big) \subset B_{\epsilon} (\gamma)
\end{align*} for every $\gamma \in \Gamma \smallsetminus F$.
\end{lem}
We will need the following convergence group analog of a well-known result of Abels$-$Margulis$-$Soifer, which, in particular, states that elements in a strongly irreducible linear group are uniformly close to proximal elements (see \cite{AMS}). We actually state a slightly strengthened version of Lemma 2.4 in \cite{BCZZ1}, but its proof is essentially exactly the same. 
\begin{lem} [Lemma 2.4 of \cite{BCZZ1}]
\label{ConvGpAMS}
Suppose $\Gamma \subset \mathrm{Homeo}(M)$ is a convergence group and $d$ is a compatible metric on $\Gamma \sqcup M$. Then there exists $\epsilon > 0$ and a finite set $F \subset \Gamma$ with the following property: for any $\gamma \in \Gamma$, there is some $f \in F$ so that $\gamma f$ is loxodromic and 
\begin{align*}
    \min \{d((\gamma f)^{+}, (\gamma f)^{-}), d(\gamma f, (\gamma f)^{-}), d((\gamma f)^{+}, (\gamma f)^{-1}), d(\gamma f, (\gamma f)^{-1}) \} > \epsilon.
\end{align*}
\end{lem}
\subsubsection{Properties of Expanding Coarse-Cocycles and GPS Systems}
We first recall the notion of a coarse-cocycle.
\begin{defn} [See page 2 of \cite{BCZZ1}]
Suppose $\Gamma \subset \mathrm{Homeo}(M)$ is a convergence group. A function $\sigma : \Gamma \times M \rightarrow \mathbb{R}$ is called a $\kappa$-$\textit{coarse-cocycle}$ if:
\begin{itemize}
    \item[(1)] For every $\gamma \in \Gamma$, the function $\sigma(\gamma, \cdot) : M \rightarrow \mathbb{R}$ is $\kappa$-$\textit{coarsely continuous}$; that is, for all $x_{0} \in M$ 
    \begin{align*}
        \limsup_{x \to x_{0}} |\sigma(\gamma, x) - \sigma (\gamma, x_{0})| \leq \kappa.
    \end{align*}
    \item[(2)] The function $\sigma$ satisfies a coarse version of the cocycle identity: for all $\gamma_{1}, \gamma_{2} \in \Gamma$ and $x \in M$, we have
    \begin{align*}
        \Big| \sigma (\gamma_{1} \gamma_{2}, x) - \Big( \sigma (\gamma_{1}, \gamma_{2}x) + \sigma(\gamma_{2},x) \Big) \Big| \leq \kappa.
    \end{align*}
\end{itemize}
\end{defn}
Notice that a $0$-coarse-cocycle is simply a continuous cocycle. We will be interested in coarse-cocycles satisfying certain additional properties.
\begin{defn} [Definition 1.2 of \cite{BCZZ1}]
Suppose $\Gamma \subset \mathrm{Homeo}(M)$ is a non-elementary convergence group and $d$ is a compatible metric on $\Gamma \sqcup M$. A coarse-cocycle $\sigma : \Gamma \times M \rightarrow \mathbb{R}$ is said to be:
\begin{itemize}
    \item[(1)] $\textit{Proper}$ if $\sigma(\gamma_{n}, \gamma_{n}^{+}) \rightarrow \infty$ whenever $\{\gamma_{n}\} \subset \Gamma$ is a sequence of pairwise distinct loxodromic elements so that the sequence of pairs of repelling/attracting fixed points satisfies
    \begin{align*}
        \liminf_{n \to \infty} d(\gamma_{n}^{-}, \gamma_{n}^{+}) > 0.
    \end{align*}
    \item[(2)] $\textit{Expanding}$ if it is proper and for every $\gamma \in \Gamma$, there is a number $||\gamma||_{\sigma} \in \mathbb{R}$, called the $\sigma$-$\textit{magnitude}$ of $\gamma$ satisfying the following: for every $\epsilon > 0$, there exists $C = C(\epsilon) > 0$ so that if $x \in M$, $\gamma \in \Gamma$, and $d(x, \gamma^{-1}) > \epsilon$, then
    \begin{align*}
        \big| \sigma(\gamma, x) - ||\gamma||_{\sigma} \big| \leq C.
    \end{align*}
\end{itemize}
\end{defn}
Given an expanding coarse-cocycle $\sigma : \Gamma \times M \rightarrow \mathbb{R}$, the $\sigma$-$\textit{Poincar\'e series}$ is 
\begin{align*}
Q_{\sigma}(s) = \sum_{\gamma \in \Gamma} e^{-s ||\gamma||_{\sigma}} \in [0, \infty]
\end{align*}
and the $\sigma$-$\textit{critical exponent}$ is the abscissa of convergence of this series, that is,
\begin{align*}
    \delta_{\sigma}(\Gamma) = \inf \{s > 0 : Q_{\sigma}(s) < \infty \} \in [0, \infty].
\end{align*} 
We will make frequent use of the following result in \cite{BCZZ1}, which establishes a number of useful properties related to expanding coarse-cocycles.
\begin{prop} [Proposition 3.3 of \cite{BCZZ1}]
\label{CocycleProperties}
Suppose $\Gamma \subset \mathrm{Homeo}(M)$ is a convergence group, $d$ is a compatible metric on $\Gamma \sqcup M$, and $\sigma$ is an expanding $\kappa$-coarse-cocycle. Then:
\begin{itemize}
    \item[(1)] For any finite subset $F \subset \Gamma$, there exists $C > 0$ such that: if $\gamma \in \Gamma$ and $f \in F$, then
    \begin{align*}
        ||\gamma||_{\sigma} - C \leq ||\gamma f||_{\sigma} \leq ||\gamma||_{\sigma} + C \ \ \ \mathrm{and} \ \ \ ||\gamma||_{\sigma} - C \leq ||f \gamma||_{\sigma} \leq ||\gamma||_{\sigma} + C.
    \end{align*}
    \item[(2)] $\lim_{n \to \infty} ||\gamma_{n}||_{\sigma} = \infty$ for every escaping sequence $\{\gamma_{n}\} \subset \Gamma$.
   \item[(3)] For any $\epsilon > 0$, there exists $C > 0$ such that: if $\alpha, \beta \in \Gamma$ and $d(\alpha^{-1}, \beta) \geq \epsilon$, then
    \begin{align*}
        ||\alpha||_{\sigma} + ||\beta||_{\sigma} - C \leq ||\alpha \beta||_{\sigma} \leq ||\alpha||_{\sigma} + ||\beta||_{\sigma} + C.
    \end{align*}
\end{itemize}
\end{prop}
We will use the following notion of a coarse Gromov--Patterson--Sullivan system to apply our construction of a Bishop--Jones semigroup in the abstract convergence group setting to the particular case of $P_{\theta}$-transverse groups acting on their limit sets.
\begin{defn} [Definition 1.7 of \cite{BCZZ1}]
Suppose $\Gamma \subset \mathrm{Homeo}(M)$ is a non-elementary convergence group and let $M^{(2)} = \{(x,y) \in M^2 : x \neq y \}$. We say that $(\sigma, \bar{\sigma}, G)$ is a $\kappa$-$\textit{coarse Gromov--Patterson--Sullivan system}$ (or $\textit{GPS system}$ for short) if $\sigma, \bar{\sigma} : \Gamma \times M \rightarrow \mathbb{R}$ are proper $\kappa$-coarse-cocycles, $G : M^{(2)} \rightarrow \mathbb{R}$ is a locally bounded function, and 
\begin{align*}
    \Big| \Big( \bar{\sigma}(\gamma, x) + \sigma(\gamma, y) \Big) - (G(\gamma x, \gamma y) - G(x,y) \Big) \Big| \leq \kappa
\end{align*} for all $\gamma \in \Gamma$ and distinct $x,y \in M$. When $\kappa = 0$ and $G$ is continuous, we call $(\sigma, \bar{\sigma}, G)$ a $\textit{continuous GPS system}$.
\end{defn}
The following result shows, among other things, that if $(\sigma, \bar{\sigma}, G)$ is a coarse GPS system, then $\sigma$ and $\bar{\sigma}$ are expanding coarse-cocycles.
\begin{prop} [Parts (1) and (2) of Proposition 3.4 of \cite{BCZZ1}]
\label{GPSProperties}
If $\Gamma \subset \mathrm{Homeo}(M)$ is a convergence group and $(\sigma, \bar{\sigma}, G)$ is a coarse GPS system, then: 
\begin{itemize}
    \item[(1)] The coarse-cocycles $\sigma$ and $\bar{\sigma}$ are expanding.
    \item[(2)] There exists $C > 0$ such that 
    \begin{align*}
        ||\gamma^{-1}||_{\bar{\sigma}} - C \leq ||\gamma||_{\sigma} \leq ||\gamma^{-1}||_{\bar{\sigma}} + C
    \end{align*} for all $\gamma \in \Gamma$.
\end{itemize}
\end{prop}
\begin{exmp}
\label{GromovHypEx}
To better illustrate this setting, we describe one very important motivating example, which will allow our general result to be applied to prove Theorem $\ref{NoGapRank1}$. Let $X$ be a proper geodesic Gromov hyperbolic space and $\Gamma \subset \mathrm{Isom}(X)$ a discrete group of isometries of $X$. Then $\Gamma$ acts as a convergence group on the Gromov boundary $\partial X$ of $X$ (see \cite{F} or \cite{T}). Fix a base point $o \in X$. Then we may define, for each $x \in X$, a $\textit{Busemann function}$ $b_{x} : X \rightarrow \mathbb{R}$ via
\begin{align*}
    b_{x}(q) = \limsup_{p \to x} d(p,q)-d(p,o).
\end{align*} The $\textit{Busemann coarse-cocycle}$ $\beta : \Gamma \times \partial X \rightarrow \mathbb{R}$ is then defined by 
\begin{align*}
    \beta(\gamma,x) = b_{x}(\gamma^{-1}(o)).
\end{align*} In general, this is only a coarse-cocycle. The classical $\textit{Gromov product}$ is the map $G : \partial X^{(2)} \rightarrow \mathbb{R}$ defined by
\begin{align*}
G(x,y) = \limsup_{p \to x, q \to y} d(o,p) + d(o,q) - d(p,q).
\end{align*} We then have that $(\beta, \beta, G)$ is a coarse GPS system (which is not necessarily continuous). Moreover, one can take the $\beta$-magnitude function to be $||\gamma||_{\beta} = d(o,\gamma o)$, hence $\delta_{\beta}(\Gamma)$ is precisely the usual critical exponent $\delta(\Gamma)$ of the Poincar\'e series
\begin{align*}
    Q(s) = \sum_{\gamma \in \Gamma} e^{-s d(o,\gamma o)}.
\end{align*} If $X$ is CAT$(-1)$ (for instance, if $X$ is a rank one symmetric space), then $\beta$ is a continuous cocycle and $(\beta, \beta, G)$ is a continuous GPS system.
\end{exmp}
\subsection{Shadows and Coarse Patterson--Sullivan Measures} We now recall the definition of shadows from \cite{BCZZ1}. They will play an important role in our construction of Bishop--Jones semigroups.
\begin{defn} [Shadow; see Definition 1.15 of \cite{BCZZ1}]
Suppose $\Gamma \subset \mathrm{Homeo}(M)$ is a non-elementary convergence group and $d$ is a compatible metric on $\Gamma \sqcup M$. Given $\epsilon > 0$ and $\gamma \in \Gamma$, the associated $\textit{shadow}$ is defined by
\begin{align*}
\mathcal{S}_{\epsilon} (\gamma) := \gamma \big(M \smallsetminus B_{\epsilon}(\gamma^{-1}) \big).
\end{align*}
\end{defn} Readers familiar with shadows in the classical setting of Gromov hyperbolic spaces may find this definition strange at first. However, Proposition $5.5$ of \cite{BCZZ1} shows that when $X$ is a proper geodesic Gromov hyperbolic space and $\Gamma \subset \mathrm{Isom}(X)$ is discrete, then every shadow, in the sense of the above definition, in the Gromov boundary $\partial X$ is contained in a classically defined shadow; likewise any classically defined shadow in $\partial X$ is contained in a shadow in the convergence group sense. \par 
The following are some important properties of shadows which will be needed later on in the paper.
\begin{prop} [Proposition 5.1 of \cite{BCZZ1}]
\label{ShadowProperties}
If $\epsilon > 0$ and $\sigma : \Gamma \times M \rightarrow \mathbb{R}$ is an expanding coarse-cocycle, then:
\begin{itemize}
    \item[(1)] There exists $C_{1} > 0$ such that if $x \in \mathcal{S}_{\epsilon}(\gamma)$, then
    \begin{align*}
        ||\gamma||_{\sigma} - C_{1} \leq \sigma(\gamma, \gamma^{-1}(x)) \leq ||\gamma||_{\sigma} + C_{1}.
    \end{align*}
    \item[(2)] If $\{\gamma_{n}\} \subset \Gamma$ is an escaping sequence, then 
    \begin{align*}
        \lim_{n \to \infty} \mathrm{diam} \ \mathcal{S}_{\epsilon}(\gamma_{n}) = 0 \ \ \mathrm{and} \ \ \lim_{n \to \infty} \inf_{x \in \mathcal{S}_{\epsilon}(\gamma_{n})} d(\gamma_{n},x) = 0,
    \end{align*} where the diameter is with respect to $d$. In particular, the Hausdorff distance with respect to $d$ between the sets $\{\gamma_{n}\}$ and $\mathcal{S}_{\epsilon}(\gamma_{n})$ converges to zero.
    \item[(3)] There exists $C_{2} > 0$ such that: if $\alpha, \beta \in \Gamma$, $||\alpha||_{\sigma} \leq ||\beta||_{\sigma}$, and  $\mathcal{S}_{\epsilon}(\alpha) \cap \mathcal{S}_{\epsilon}(\beta) \neq \emptyset$,
then
    \begin{align*}
        ||\alpha^{-1} \beta||_{\sigma} + ||\alpha||_{\sigma} - C_{2} \leq ||\beta||_{\sigma} \leq ||\alpha^{-1} \beta||_{\sigma} + ||\alpha||_{\sigma} + C_{2}.
    \end{align*}
\end{itemize}
\end{prop} 
Shadows are closely related to the limit set of $\Gamma$. We begin with the following definition.
\begin{defn} [Definition $5.2$ of \cite{BCZZ1}]
Given $\epsilon > 0$, the $\epsilon$-$\textit{uniform conical limit set}$, denoted $\Lambda^{\mathrm{con}}_{\epsilon}(\Gamma)$, is the set of points $x \in M$ such that there exist $a,b \in M$ and a sequence of elements $\{\gamma_{n}\}$ in $\Gamma$ so that $d(a,b) \geq \epsilon$, $\lim_{n \to \infty} \gamma_{n}x = b$, and $\lim_{n \to \infty} \gamma_{n}y = a$ for all $y \in M \smallsetminus \{x\}$.
\end{defn} By definition, we have 
\begin{align*}
\Lambda^{\mathrm{con}}(\Gamma) = \bigcup_{\epsilon > 0} \Lambda_{\epsilon}^{\mathrm{con}}(\Gamma) = \bigcup_{n=1}^{\infty} \Lambda_{\frac{1}{n}}^{\mathrm{con}} (\Gamma).
\end{align*} 
Moreover, these limit sets are $\Gamma$-invariant; see Observation $5.3$ of \cite{BCZZ1} for a proof. Shadows are related to uniformly conical limit sets by the following lemma. \\
\begin{lem} [Lemma $5.4$ of \cite{BCZZ1}] ~\
\label{ShadowsAndLimitSets} 
\begin{itemize}
    \item[(1)] If $x \in \Lambda_{\epsilon}^{\mathrm{con}}(\Gamma)$ and $0 < \epsilon' < \epsilon$, then there exists an escaping sequence $\{\gamma_{n}\} \subset \Gamma$ such that $x \in \bigcap_{n} \mathcal{S}_{\epsilon'}(\gamma_{n})$.
   \item[(2)] If there exists an escaping sequence $\{\gamma_{n}\} \subset \Gamma$ such that $x \in \bigcap_{n} \mathcal{S}_{\epsilon}(\gamma_{n})$, then $x \in \Lambda_{\epsilon}^{\mathrm{con}}(\Gamma)$. 
\end{itemize}
\end{lem}
We now consider the notion of Patterson--Sullivan measures associated to a coarse-cocycle.

\begin{defn} [Definition 1.1 of \cite{BCZZ1}]
Suppose $\Gamma \subset \mathrm{Homeo}(M)$ is a convergence group and $\sigma : \Gamma \times M \rightarrow \mathbb{R}$ is a coarse-cocycle. A probability measure $\mu$ is a $C$-$\textit{coarse}$ $\sigma$-$\textit{Patterson--Sullivan measure of dimension}$ $\delta$ if, for every $\gamma \in \Gamma$, the measures $\mu$ and $\gamma_{*}\mu$ are absolutely continuous with respect to each other and 
\begin{align*}
    e^{-C-\delta \sigma(\gamma^{-1}, \cdot)} \leq \frac{d(\gamma_{*}\mu)}{d\mu} \leq e^{C-\delta \sigma(\gamma^{-1}, \cdot)},
\end{align*} $\mu$-almost everywhere.
\end{defn}
The following result shows that coarse Patterson--Sullivan measures exist in the critical dimension.
\begin{thm} [Theorem 4.1 of \cite{BCZZ1}]
\label{PSmeasure}
If $\sigma$ is an expanding $\kappa$-coarse-cocycle for a non-elementary convergence group $\Gamma \subset \mathrm{Homeo}(M)$ and $\delta := \delta_{\sigma}(\Gamma) < \infty$, then there exists a $2 \kappa \delta$-coarse $\sigma$-Patterson--Sullivan measure of dimension $\delta$, which is supported on the limit set $\Lambda (\Gamma)$.
\end{thm}
\section{A Bishop--Jones Semigroup for Convergence Groups with Expanding Coarse-Cocycles}
\label{ConstructionSection}
\subsection{The Main Construction}
Let $\Gamma \subset \mathrm{Homeo}(M)$ be a discrete non-elementary convergence group equipped with an expanding coarse-cocycle $\sigma : \Gamma \times M \rightarrow \mathbb{R}$. For $n \geq 0$, define
\begin{align*}
    \mathcal{A}_{\sigma, n} := \{ \gamma \in \Gamma \ : \ n \leq ||\gamma||_{\sigma} < n+1\}.
\end{align*} Given a finite subset $S \subset \Gamma$, define 
\begin{align*}
    S^{0} := \{\mathrm{id}\} \ \ \mathrm{and} \ \  S^{m} := \{\gamma_{1} \cdots \gamma_{m} : \gamma_{i} \in S \ \ \mathrm{for} \ \ 1 \leq i \leq m \} \ \mathrm{for} \ m \geq 1.
\end{align*}
For $\gamma \in \bigcup_{m=0}^{\infty} S^{m}$, let 
\begin{align*}
|\gamma|_{S} := \min \{k \geq 1 : \gamma = \gamma_{1} \cdots \gamma_{k}, \ \mathrm{where} \ \gamma_{i} \in S \ \mathrm{for \ all} \ 1 \leq i \leq k \},
\end{align*} if $\gamma \neq \mathrm{id}$ and set $|\mathrm{id}|_{S} := 0$. Our main result is the following theorem.
\begin{thm} [Main Theorem]
\label{MainThm}
For every $0 < \delta < \delta_{\sigma}(\Gamma)$, there exists a free finitely generated semigroup $\mathcal{T} = \mathcal{T}_{\delta} \subset \Gamma$ with finite generating set $S \subset \Gamma$, which we call a $\mathrm{Bishop}$--$\mathrm{Jones \ semigroup}$, with the following properties:
\begin{itemize}
    \item[(1)] Its critical exponent satisfies \begin{align*}
        \delta_{\sigma}(\mathcal{T}) \geq \delta.
    \end{align*}
    \item[(2)] Given any expanding coarse-cocycle $\phi : \Gamma \times M \rightarrow \mathbb{R}$, there exist constants $B_{1} = B_{1}(\phi) > 1$ and $b_{1} = b_{1}(\phi) > 0$ so that
\begin{align*}
   \frac{1}{B_{1}} |\gamma|_{S} - b_{1} \leq ||\gamma||_{\phi} \leq B_{1} |\gamma|_{S} + b_{1}
\end{align*} for all $\gamma \in \mathcal{T}$. 
    \item[(3)] With $B_{1} = B_{1}(\sigma)$ as above, we have
    \begin{align*}
        \frac{1}{B_{1}} \log (\# S) \leq \delta_{\sigma}(\mathcal{T}) \leq B_{1} \log (\# S).
    \end{align*} In particular, $\delta_{\sigma}(\mathcal{T}) < \infty$.
    \item[(4)] The critical exponent of a Bishop--Jones semigroup is strictly less than that of the ambient group, that is 
    \begin{align*}
        \delta_{\sigma}(\mathcal{T}) < \delta_{\sigma}(\Gamma).
    \end{align*}
    \item[(5)] Given any two expanding coarse-cocycles $\phi_{1}, \phi_{2} : \Gamma \times M \rightarrow \mathbb{R}$, there exist constants $B_{2} = B_{2}(\phi_{1}, \phi_{2}) > 1$ and $b_{2} = b_{2}(\phi_{1}, \phi_{2}) > 0$ so that
\begin{align*}
    \frac{1}{B_{2}} ||\gamma||_{\phi_{2}} - b_{2} \leq ||\gamma||_{\phi_{1}} \leq B_{2}||\gamma||_{\phi_{2}} + b_{2}
\end{align*} for all $\gamma \in \mathcal{T}$.
\end{itemize}
\end{thm}
Property (1) is established in Corollary \ref{CritExpApprox}. Property (2) is proved in Lemma \ref{TreeMag1}. Notice that property (5) is an immediate consequence of property (2). Property (3) is established in Proposition \ref{LogBounds}. By this property, to establish property (4) it suffices to consider the case when $\delta_{\sigma}(\Gamma) < \infty$. This is the content of Theorem \ref{SemigroupExpGap}.
\par 
In section \ref{PropertiesSubsection}, we provide the proofs of properties (1), (2), and (5), as well as discussing the relation between Bishop--Jones semigroups and the conical limit set of $\Gamma$. Section \ref{ExpGapSubsection} is devoted to the proofs of properties (3) and (4).
\par 
The heart of the argument is the following proposition, whose proof occupies the remainder of this section. It is the key technical result needed to establish Theorem \ref{MainThm}. We emphasize that not only the statement of the proposition, but also its proof, will be used later on. Henceforth we fix a compatible metric $d$ on $\Gamma \sqcup M$.
\begin{prop}
\label{BJtree}
For every $0 < \delta < \delta_{\sigma}(\Gamma)$, there exist a finite subset $S \subset \Gamma$, $D_{0} > 0$, and $t_{0} > 0$ so that: for any $m \geq 0$, if $\gamma \in S^{m}$, then the set $\gamma \cdot S \subset S^{m+1}$ has the following properties:
\begin{itemize}
    \item[(1)] If $\eta \in \gamma \cdot S$, then $\mathcal{S}_{\frac{t_{0}}{4}}(\eta) \subset \mathcal{S}_{\frac{t_{0}}{2}}(\gamma)$.
    \item[(2)] The sets $\{\mathcal{S}_{\frac{t_{0}}{4}}(\eta)\}_{\eta \in \gamma \cdot S}$ are pairwise disjoint.
    \item[(3)] If $\eta \in \gamma \cdot S$, then $||\eta^{-1} \gamma||_{\sigma} \leq D_{0}$.
    \item[(4)] $\sum_{\eta \in \gamma \cdot S} e^{-\delta||\eta||_{\sigma}} \geq e^{-\delta ||\gamma||_{\sigma}}$. 
\end{itemize} Moreover, the resulting subset $\mathcal{T} = \mathcal{T}_{\delta}$ of $\Gamma$ defined by
\begin{align*}
    \mathcal{T} := \bigcup_{m=0}^{\infty} S^{m}
\end{align*}
is a free subsemigroup of $\Gamma$ with finite generating set $S$.
\end{prop}
Before we can give the proof, we need to establish a few technical lemmas. 
\begin{lem}
\label{ColoringLem}
For any $C > 0$, there exists an integer $L = L(C) \geq 1$ and a partition $\Gamma = P_{1} \sqcup \cdots \sqcup P_{L}$ of $\Gamma$ so that for any $1 \leq j \leq L$ and any pair of distinct elements $\alpha, \beta \in P_{j}$, we have $||\alpha^{-1} \beta||_{\sigma} \geq C$.
\end{lem}
\begin{proof}
Fix an enumeration $\Gamma = \{\gamma_{n}\}_{n \geq 1}$ of the elements of $\Gamma$. By part (2) of Proposition \ref{CocycleProperties}, we have that $||\gamma_{n}||_{\sigma} \rightarrow \infty$ as $n \rightarrow \infty$. Thus there exists an integer $N \geq 1$ so that $||\gamma_{n}||_{\sigma} \geq C$ for all $n \geq N$. Define a graph $\mathcal{G}$ having vertex set the elements of $\Gamma$ as follows. There is an edge between two elements $\alpha$ and $\beta$ of $\Gamma$ if and only if $\alpha \neq \beta$ and 
\begin{align*}
    \alpha^{-1} \beta \in \{\gamma_{1}, \gamma_{2}, \dots, \gamma_{N-1}, \gamma_{1}^{-1}, \gamma_{2}^{-1}, \dots, \gamma_{N-1}^{-1} \}.
\end{align*} Thus the degree of the graph $\mathcal{G}$ is bounded above by $2N-2$. Applying a greedy coloring, we see that $\mathcal{G}$ can be colored by $L := 2N-1$ colors so that adjacent vertices have different colors. By construction, this coloring corresponds exactly to a partition $\Gamma = P_{1} \sqcup \cdots \sqcup P_{L}$ of $\Gamma$ so that for any $1 \leq j \leq L$, if $\alpha, \beta \in P_{j}$ are distinct elements, then 
\begin{align*}
    \alpha^{-1} \beta \notin \{\gamma_{1}, \gamma_{2}, \dots, \gamma_{N-1}, \gamma_{1}^{-1}, \gamma_{2}^{-1}, \dots, \gamma_{N-1}^{-1} \},
\end{align*} which implies that $||\alpha^{-1} \beta||_{\sigma} \geq C$, as desired.
\end{proof}
\begin{lem}
\label{ConvGpLem}
Fix $x \in M$, $f \in \Gamma$, and $0 < \epsilon' < \epsilon$. There exists $N = N(x,f,\epsilon,\epsilon') \geq 1$ so that for all $n \geq N$, if $\gamma \in \mathcal{A}_{\sigma, n} \cap B_{\epsilon'}(x)$, then $\gamma f \in B_{\epsilon}(x)$.
\end{lem}
\begin{proof}
Suppose not. Then there is some sequence of integers $n_{j} \rightarrow \infty$ and a sequence $\{\gamma_{n_{j}}\} \subset \Gamma$ so that $\gamma_{n_{j}} \in \mathcal{A}_{\sigma, n_{j}} \cap B_{\epsilon'}(x)$ but $\gamma_{n_{j}}f \notin B_{\epsilon}(x)$ for all $j \geq 1$. By passing to a subsequence $\{\gamma_{n_{j_{k}}}\}$ of $\{\gamma_{n_{j}}\}$, we have that $\gamma_{n_{j_{k}}} \rightarrow a \in \overline{B_{\epsilon'}(x)} \subset B_{\epsilon}(x)$ and $\gamma_{n_{j_{k}}}^{-1} \rightarrow b \in M$ as $k \rightarrow \infty$. By part (1) of Proposition 2.3 of \cite{BCZZ1}, it follows that $\gamma_{n_{j_{k}}}|_{M \smallsetminus \{b\}}$ converges locally uniformly to $a$. Further, the proof of that proposition then implies that $\gamma_{n_{j_{k}}}|_{(\Gamma \sqcup M) \smallsetminus \{b\}}$ converges locally uniformly to a. Since $f \neq b$ (as $f \in \Gamma$), this implies that $\gamma_{n_{j_{k}}}f \in B_{\epsilon}(x)$ for some $k$ sufficiently large, which is a contradiction.
\end{proof}
The following lemma is immediate from part (3) of Proposition \ref{CocycleProperties}, but we record it below as it will be convenient later.
\begin{lem}
\label{TreeLem1}
For every $\epsilon > 0$, there exists $A = A(\epsilon) > 0$ so that if $\alpha, \beta \in \Gamma$ and $d(\alpha^{-1}, \beta) \geq \epsilon$, then 
\begin{align*}
e^{-||\alpha \beta||_{\sigma}} \geq A e^{-||\alpha||_{\sigma}}  e^{-||\beta||_{\sigma}}. 
\end{align*}
\end{lem}
\begin{proof}
By part (3) of Proposition \ref{CocycleProperties}, there exists a constant $C = C(\epsilon) > 0$ so that if $\alpha, \beta \in \Gamma$ satisfy $d(\alpha^{-1}, \beta) \geq \epsilon$, then
\begin{align*}
    ||\alpha \beta||_{\sigma} \leq ||\alpha||_{\sigma} + ||\beta||_{\sigma} + C.
\end{align*} Hence
\begin{align*}
   e^{-||\alpha \beta||_{\sigma}} \geq e^{-C} e^{-||\alpha||_{\sigma}} e^{-||\beta||_{\sigma}} = A e^{-||\alpha||_{\sigma}} e^{-||\beta||_{\sigma}},
\end{align*} where we define $A := e^{-C}$.
\end{proof}
\begin{lem}
\label{TreeLem2}
For every $\epsilon > 0$ and $M \in \mathbb{N}$, there exists a constant $C = C(\epsilon, M) > 0$ so that if $\alpha, \beta \in \Gamma$ satisfy
\begin{align*}
0 \leq ||\beta||_{\sigma} - ||\alpha||_{\sigma} \leq M \ \ \mathrm{and} \ \ ||\alpha^{-1} \beta||_{\sigma} > C,
\end{align*} then
\begin{align*}
\mathcal{S}_{\epsilon}(\alpha) \cap \mathcal{S}_{\epsilon}(\beta) = \emptyset.
\end{align*}
\end{lem}
\begin{proof}
Suppose not. Then there exists $\epsilon > 0$, an integer $M \geq 1$, and sequences $\{\alpha_{n}\}, \{ \beta_{n}\} \subset \Gamma$ so that 
\begin{align*}
0 \leq ||\beta_{n}||_{\sigma} - ||\alpha_{n}||_{\sigma} \leq M \ \ \mathrm{and} \ \ ||\alpha_{n}^{-1} \beta_{n}||_{\sigma} > n,
\end{align*}
but
\begin{align*}
\mathcal{S}_{\epsilon}(\alpha_{n}) \cap \mathcal{S}_{\epsilon}(\beta_{n}) \neq \emptyset.
\end{align*} Part (3) of Proposition \ref{ShadowProperties} then implies that there exists a constant $C > 0$, depending only on $\epsilon$, so that for all $n \geq 1$
\begin{align*}
||\alpha_{n}^{-1} \beta_{n}||_{\sigma} \leq ||\beta_{n}||_{\sigma} - ||\alpha_{n}||_{\sigma} + C \leq M+C,
\end{align*} which is impossible for all $n > M+C$, giving the desired contradiction.
\end{proof}
\begin{lem}
\label{TreeLem3}
Fix $t > 0$. If $\alpha, \beta \in \Gamma$ are such that $d(\alpha, \beta) \geq t$ and $||\beta||_{\sigma}$ is sufficiently large (depending on $t$), then
\begin{align*}
    \alpha \mathcal{S}_{\frac{t}{4}} (\alpha^{-1} \beta) \subset \alpha \mathcal{S}_{\frac{t}{2}}(\alpha^{-1}) \cap \mathcal{S}_{\frac{t}{8}}(\beta).
\end{align*}
\end{lem}
\begin{proof}
Suppose not. Then there exist sequences $\{ \alpha_{n}\}, \{ \beta_{n}\} \subset \Gamma$ with $d(\alpha_{n}, \beta_{n}) \geq t$, $||\beta_{n}||_{\sigma} \rightarrow \infty$, and such that there exist
\begin{align}
\label{Incl1}
    x_{n} \in \alpha_{n} \mathcal{S}_{\frac{t}{4}} (\alpha_{n}^{-1} \beta_{n}) \smallsetminus \Big( \alpha_{n} \mathcal{S}_{\frac{t}{2}}(\alpha_{n}^{-1}) \cap \mathcal{S}_{\frac{t}{8}}(\beta_{n}) \Big)
\end{align} for all $n \geq 1$. Passing to subsequences, we may assume without loss of generality that $\alpha_{n}^{\pm 1} \rightarrow a^{\pm}$, $\beta_{n}^{\pm 1} \rightarrow b^{\pm}$, and $x_{n} \rightarrow x$, where $a^{\pm} \in \Gamma \sqcup M$ and $b^{\pm}, x \in M$. Write 
\begin{align*}
x_{n} = \beta_{n} y_{n} \in \alpha_{n}\mathcal{S}_{\frac{t}{4}}(\alpha_{n}^{-1} \beta_{n}),
\end{align*}
where $d(y_{n}, \beta_{n}^{-1} \alpha_{n})  \geq t/4$. Since $d(\alpha_{n}, \beta_{n}) \geq t$, we have $\beta_{n}^{-1} \alpha_{n} \rightarrow b^{-}$, and therefore $x_{n} = \beta_{n} y_{n} \rightarrow b^{+}$. Thus $x = b^{+}$ and $d(x_{n}, \beta_{n}) \rightarrow 0$. So for all $n$ sufficiently large, we have $d(x_{n}, \alpha_{n}) \geq t/2$, that is
\begin{align}
\label{Incl2}
    x_{n} \in M \smallsetminus B_{\frac{t}{2}}(\alpha_{n}) = \alpha_{n} \mathcal{S}_{\frac{t}{2}}(\alpha_{n}^{-1})
\end{align}
for $n$ large enough. But since
\begin{align*}
d(y_{n}, \beta_{n}^{-1} \alpha_{n}) \geq \frac{t}{4} \ \ \mathrm{and} \ \  \beta_{n}^{-1} \alpha_{n} \rightarrow b^{-},
\end{align*}
we have 
\begin{align*}
\liminf_{n \to \infty} d(y_{n}, \beta_{n}^{-1}) \geq \frac{t}{4} \implies y_{n} \in M \smallsetminus B_{\frac{t}{8}}(\beta_{n}^{-1}),
\end{align*} for all $n$ sufficiently large. Hence also
\begin{align}
\label{Incl3}
x_{n} = \beta_{n}y_{n} \in \mathcal{S}_{\frac{t}{8}}(\beta_{n}),
\end{align} for $n$ sufficiently large. But then for $n$ large enough, (\ref{Incl2}) and (\ref{Incl3}) imply that $x_{n} \in \alpha_{n} \mathcal{S}_{\frac{t}{2}}(\alpha_{n}^{-1}) \cap \mathcal{S}_{\frac{t}{8}}(\beta_{n})$. This contradicts (\ref{Incl1}), concluding the proof.
\end{proof}
We now turn our attention to the Poincar\'e series $Q_{\sigma}(s) = \sum_{\gamma \in \Gamma} e^{-s||\gamma||_{\sigma}}$. Fix $0 < \delta < \delta + \epsilon < \delta_{\sigma}(\Gamma)$. By definition, we have
\begin{align*}
    Q_{\sigma}(\delta + \epsilon) = \sum_{\gamma \in \Gamma} e^{-(\delta + \epsilon) ||\gamma||_{\sigma}} = \infty.
\end{align*} Since $\Gamma \sqcup M$ is compact, for each $n \geq 1$, there exists $x_{n} \in M$ so that 
\begin{align*}
    \sum_{\eta \in \Gamma \cap B_{\frac{1}{n}}(x_{n})} e^{-(\delta + \epsilon) ||\eta||_{\sigma}} = \infty.
\end{align*} Passing to a subsequence, we can assume that $x_{n} \rightarrow x \in \Lambda(\Gamma) \subset M$. It follows that for all $t > 0$, we have
\begin{align*}
    \sum_{\gamma \in \Gamma \cap B_{t}(x)} e^{-(\delta + \epsilon) ||\gamma||_{\sigma}} = \infty.
\end{align*}
\begin{lem}
\label{TreeLem4}
For any $y \in \Lambda (\Gamma)$ and $t > 0$, we have
\begin{align*}
    \limsup_{n \to \infty} \sum_{\eta \in \mathcal{A}_{\sigma, n} \cap B_{t}(y)} e^{-\delta ||\eta||_{\sigma}} = \infty.
\end{align*}
\end{lem}
\begin{proof}
The general strategy is as in the proof of Lemma 8.7 of \cite{CZZ1}, with some technical modifications. We first prove the lemma in the case when $y = \gamma x \in \Gamma x$. By part (1) of Proposition \ref{CocycleProperties}, there is a constant $C = C(\gamma) > 0$ so that for all $\xi \in \Gamma$, 
\begin{align*}
    ||\xi||_{\sigma} - C &\leq ||\gamma \xi||_{\sigma} \leq ||\xi||_{\sigma} + C, \ \mathrm{hence} \\ 
    e^{-C} e^{-||\xi||_{\sigma}} &\leq e^{-||\gamma \xi||_{\sigma}} \leq e^{C} e^{-||\xi||_{\sigma}}.
\end{align*} Setting $D := e^{C} > 1$, we have 
\begin{align*}
e^{-||\gamma \xi||_{\sigma}} \geq \frac{1}{D} e^{-||\xi||_{\sigma}}
\end{align*}
for all $\xi \in \Gamma$. Since the action of $\gamma$ on $\Gamma \sqcup M$ is continuous, $\gamma^{-1}(B_{\frac{t}{2}}(\gamma x))$ is an open subset of $\Gamma \sqcup M$ containing $x$. Therefore there exists some $t' > 0$ so that $B_{t'}(x) \subset \gamma^{-1}(B_{\frac{t}{2}}(\gamma x))$, and so
\begin{align*}
B_{t'}(x) \cap \Gamma \subset \gamma^{-1}(B_{\frac{t}{2}}(\gamma x)) \cap \Gamma = \gamma^{-1} (\Gamma \cap B_{\frac{t}{2}} (y)).
\end{align*} It follows that
\begin{align}
\label{OrbitIneq1}
    \sum_{\eta \in \Gamma \cap B_{\frac{t}{2}} (y)} e^{-(\delta + \epsilon)||\eta||_{\sigma}} &\geq \sum_{\eta \in \gamma (\Gamma \cap B_{t'}(x))} e^{-(\delta + \epsilon)||\eta||_{\sigma}} 
   = \sum_{\xi \in \Gamma \cap B_{t'}(x)} e^{-(\delta + \epsilon)||\gamma \xi||_{\sigma}} \\
   \nonumber &\geq \frac{1}{D^{\delta + \epsilon}} \sum_{\xi \in \Gamma \cap B_{t'}(x)} e^{-(\delta + \epsilon) ||\xi||_{\sigma}} = \infty.
\end{align} If
\begin{align}
\label{OrbitIneq2}
\limsup_{n \to \infty} \sum_{\eta \in \mathcal{A}_{\sigma, n} \cap B_{\frac{t}{2}}(y)} e^{-\delta ||\eta||_{\sigma}} < \infty,
\end{align} then there exists a constant $M' > 0$ so that 
\begin{align*}
\sum_{\eta \in \mathcal{A}_{\sigma, n} \cap B_{\frac{t}{2}}(y)} e^{-\delta ||\eta||_{\sigma}} \leq M'
\end{align*} for all $n \geq 0$. But then
\begin{align*}
\sum_{\eta \in \Gamma \cap B_{\frac{t}{2}}(y)} e^{-(\delta + \epsilon) ||\eta||_{\sigma}} &\leq \sum_{n = 0}^{\infty} \bigg( e^{-\epsilon n} \sum_{\eta \in \mathcal{A}_{\sigma, n} \cap B_{\frac{t}{2}}(y)} e^{-\delta ||\eta||_{\sigma}} \bigg) \\
&\leq M' \sum_{n=0}^{\infty} e^{-\epsilon n} < \infty,
\end{align*} which is a contradiction. Thus the lemma holds if $y \in \Gamma x$. 
\par 
Suppose now that $y \in \Lambda(\Gamma) \smallsetminus \Gamma x$. Since $\Gamma$ acts minimally on $\Lambda (\Gamma)$, there is a sequence of group elements $\{\gamma_{k}\} \subset \Gamma$ so that $x_{k} := \gamma_{k} x \rightarrow y$. Given any $t > 0$, we can choose $k$ sufficiently large so that $d(y, x_{k}) < t/2$. Therefore $B_{\frac{t}{2}}(x_{k}) \subset B_{t}(y)$, hence by (\ref{OrbitIneq1}) applied to the point $x_{k} \in \Gamma x$, we have
\begin{align*}
\sum_{\eta \in \Gamma \cap B_{t}(y)} e^{-(\delta + \epsilon) ||\eta||_{\sigma}} \geq \sum_{\eta \in \Gamma \cap B_{\frac{t}{2}}(x_{k})} e^{-(\delta + \epsilon) ||\eta||_{\sigma}} = \infty.
\end{align*} Then proceeding with the same argument beginning with (\ref{OrbitIneq2}) concludes the proof.
\end{proof}
We are now ready to give the proof of Proposition \ref{BJtree}.
\begin{proof} [Proof of Proposition \ref{BJtree}]
The proof is slightly technical and depends on carefully choosing constants so that we may apply to preceding lemmas. To make the proof easier to follow, we begin by listing these constants and describing their relations with each other.
\begin{itemize}
    \item[(i)] We let $\tilde{\epsilon} > 0$ be the constant and $F \subset \Gamma$ the finite subset of $\Gamma$ appearing in Lemma \ref{ConvGpAMS}.
    \item[(ii)] Fix a point $x \in \Lambda (\Gamma)$. Since $\Lambda(\Gamma)$ is perfect, there exists $y \in \Lambda(\Gamma) \smallsetminus \{x\}$ so that $d(y,x) < \frac{\tilde{\epsilon}}{4}$. Now set 
    \begin{align*}
        t_{0} := \frac{2}{3} d(y,x), \ s_{1} := \frac{3}{4} d(y,x), \ s_{2} := \frac{5}{4} d(y,x),
    \end{align*}
    and fix some $\epsilon_{0} < \frac{1}{12} d(y,x)$. Note that we then have
    \begin{align*}
    0 < t_{0} < s_{1} < d(y,x) < s_{2} < 2t_{0} \ \ \ \mathrm{and \ also} \ \ \ 0 < t_{0} < \frac{\tilde{\epsilon}}{4}.
    \end{align*}
    \item[(iii)] By Lemma \ref{TreeLem1}, there is a constant $0 < A = A(t_{0}) < 1$ so that if $\gamma, \eta \in \Gamma$ satisfy $d(\gamma^{-1}, \eta) \geq t_{0}$, then
    \begin{align}
    \label{TreePf1}
    e^{- ||\gamma \eta||_{\sigma}} \geq A e^{- ||\gamma||_{\sigma}} e^{- ||\eta||_{\sigma}}.
    \end{align} 
    \item[(iv)] By part  (1) of Proposition \ref{CocycleProperties}, there exists a constant $\tilde{C} = \tilde{C}(F \cup F^{-1}) > 0$ so that if $\gamma \in \Gamma$ and $f \in F \cup F^{-1}$, then
    \begin{align*}
        ||\gamma||_{\sigma} - \tilde{C} \leq ||\gamma f||_{\sigma} \leq ||\gamma||_{\sigma} + \tilde{C} \ \ \ \mathrm{and} \ \ \ ||\gamma||_{\sigma} - \tilde{C} \leq ||f \gamma||_{\sigma} \leq ||\gamma||_{\sigma} + \tilde{C}.
    \end{align*}
    \item[(v)] Lastly, we let $C = C(t_{0}/8, 2\tilde{C} + 1)$ be the constant given by Lemma \ref{TreeLem2}. 
\end{itemize}
By Lemma \ref{ColoringLem}, there exists an integer $L = L(C + 2\tilde{C}) \geq 1$ and a partition $\Gamma = P_{1} \sqcup \cdots \sqcup P_{L}$ of $\Gamma$ so that for any $1 \leq j \leq L$ and distinct elements $\alpha, \beta \in P_{j}$, we have
\begin{align}
\label{TreePf2}
||\alpha^{-1} \beta||_{\sigma} \geq C + 2 \tilde{C}.
\end{align}
Moreover, after possibly refining the partition, we may assume without loss of generality that for each set $P_{i}$ of the partition, there is a single element $f_{i} \in F$ so that the elements of $P_{i}f_{i}$ are uniformly $\tilde{\epsilon}$-loxodromic, in the sense of Lemma \ref{ConvGpAMS}.
Let $N \geq 1$ be the constant furnished by Lemma \ref{ConvGpLem} corresponding to $x \in M$, $F \subset \Gamma$, and $0 < t_{0}/2 < t_{0}$. By Lemma \ref{TreeLem4}, we have
\begin{align*}
\limsup_{n \to \infty} \sum_{\gamma \in \mathcal{A}_{\sigma, n} \cap B_{\frac{t_{0}}{2}}(x)} e^{-\delta ||\gamma||_{\sigma}} = \infty,
\end{align*} and therefore there exists $n \geq N$ so that
\begin{align*}
\sum_{\gamma \in \mathcal{A}_{\sigma, n} \cap B_{\frac{t_{0}}{2}}(x)} e^{-\delta ||\gamma||_{\sigma}} \geq \frac{L e^{\delta \tilde{C}}}{A^{\delta}}.
\end{align*}
Thus there exists $1 \leq i \leq L$ so that 
\begin{align*}
\sum_{\gamma \in P_{i} \cap \mathcal{A}_{\sigma, n} \cap B_{\frac{t_{0}}{2}}(x)} e^{-\delta ||\gamma||_{\sigma}} \geq \frac{e^{\delta \tilde{C}}}{A^{\delta}}.
\end{align*} Now let $f_{i} \in F$ be such that the elements of $P_{i}f_{i}$ are uniformly $\tilde{\epsilon}$-loxodromic. Define
\begin{align*}
S := \Big( P_{i} \cap \mathcal{A}_{\sigma, n} \cap B_{\frac{t_{0}}{2}}(x) \Big) f_{i}.
\end{align*} By choosing a possibly larger $n \geq N$, the $\sigma$-magnitudes of the elements of $S$ will be sufficiently large for Lemma \ref{TreeLem3} to hold, and furthermore Lemma \ref{TreeLem0} implies that for all $\gamma \in S \cup S^{-1}$, we have
\begin{align}
\label{inclu1}
\gamma \big( (\Gamma \sqcup M) \smallsetminus B_{\epsilon_{0}}(\gamma^{-1}) \big) \subset B_{\epsilon_{0}} (\gamma).  
\end{align} Since $\epsilon_{0} < t_{0}$, this immediately yields
\begin{align}
\label{inclu2}
\gamma \big( (\Gamma \sqcup M) \smallsetminus B_{t_{0}}(\gamma^{-1}) \big) \subset B_{t_{0}} (\gamma).  
\end{align} for all $\gamma \in S \cup S^{-1}$. 
\par 
Notice that the set $S$ contains at least two elements. Indeed, $\delta \tilde{C} > 0$ and $0 < A < 1$, so $e^{\delta \tilde{C}} / A^{\delta} > 1$, while $e^{-\delta ||\gamma||_{\sigma}} \leq e^{-\delta n} < 1$ for any $\gamma \in P_{i} \cap \mathcal{A}_{\sigma, n} \cap B_{\frac{t_{0}}{2}}(x)$. Hence the above inequality can hold only if $\#(P_{i} \cap \mathcal{A}_{\sigma, n} \cap B_{\frac{t_{0}}{2}}(x)) \geq 2$, which is equivalent to $\#S \geq 2$.
\par 
By Lemma \ref{ConvGpLem}, we have $S \subset B_{t_{0}}(x)$. If $\eta \in S$, then $\eta = \gamma f_{i}$ for some $\gamma \in  P_{i} \cap \mathcal{A}_{\sigma, n} \cap \mathcal{B}(x, \epsilon)$. Then part (1) of Proposition $\ref{CocycleProperties}$ yields
\begin{align*}
e^{-\delta ||\eta||_{\sigma}} &= e^{-\delta ||\gamma f_{i}||_{\sigma}} \geq e^{-\delta \tilde{C}} e^{-\delta ||\gamma||_{\sigma}}.
\end{align*}
Combining this with the above calculations, we obtain
\begin{align}
\label{TreePf3}
\nonumber \sum_{\eta \in S} e^{-\delta ||\eta||_{\sigma}} &= \sum_{\gamma \in P_{i} \cap \mathcal{A}_{\sigma, n} \cap B_{\frac{t_{0}}{2}}(x)} e^{-\delta ||\gamma f_{i}||_{\sigma}} \\
&\geq e^{-\delta \tilde{C}} \sum_{\gamma \in P_{i} \cap \mathcal{A}_{\sigma, n} \cap B_{\frac{t_{0}}{2}}(x)} e^{-\delta ||\gamma||_{\sigma}} \geq \frac{1}{A^{\delta}}.
\end{align}
We will now show that the set $S$ generates a finitely generated semigroup containing the identity and having all the desired properties. We claim that for any $m \geq 0$ and $\gamma \in S^{m}$, the set $\gamma \cdot S \subset S^{m+1}$ satisfies properties $(1) - (4)$ of Proposition \ref{BJtree}. The subset of $\Gamma$ defined by
\begin{align*}
    \mathcal{T} := \bigcup_{m=0}^{\infty} S^{m}
\end{align*} is a finitely generated subsemigroup of $\Gamma$, and after we establish the aforementioned properties, we will be able to explain why $\mathcal{T}$ is free (hence the sets in the above union are pairwise disjoint). Before proving by induction that $\gamma \cdot S$ has the desired properties, we need some preliminary facts.
\par 
Fix distinct elements $\eta$ and $\xi$ of $P_{i} \cap \mathcal{A}_{\sigma, n} \cap B_{\frac{t_{0}}{2}}(x)$. Then $\eta f_{i}$ and $\xi f_{i}$ are distinct elements of $S$ and without loss of generality we may assume $||\xi f_{i}||_{\sigma} \geq ||\eta f_{i}||_{\sigma}$. We have
\begin{align}
\label{TreePf5}
\nonumber || \xi f_{i} ||_{\sigma} - ||\eta f_{i}||_{\sigma}  &\leq \Big| ||\eta f_{i}||_{\sigma} - ||\eta||_{\sigma} \Big| + \Big| ||\eta||_{\sigma} - ||\xi||_{\sigma} \Big| + \Big| ||\xi||_{\sigma} - ||\xi f_{i}||_{\sigma} \Big| \\ 
 &\leq 2 \tilde{C} + 1.
\end{align} Moreover, using (\ref{TreePf2}) we obtain
\begin{align}
\label{TreePf6}
\nonumber ||(\eta f_{i})^{-1} (\xi f_{i})||_{\sigma} &= ||f_{i}^{-1} \eta^{-1} \xi f_{i}||_{\sigma} \\
\nonumber &\geq ||\eta^{-1} \xi f_{i}||_{\sigma} - \tilde{C} \\
\nonumber &\geq ||\eta^{-1} \xi||_{\sigma} - 2 \tilde{C} \\
 &\geq C.
\end{align} Combining (\ref{TreePf5}) and (\ref{TreePf6}), Lemma \ref{TreeLem2} tells us that 
\begin{align}
\label{EmptyInt}
\mathcal{S}_{\frac{t_{0}}{4}}(\eta f_{i}) \cap \mathcal{S}_{\frac{t_{0}}{4}}(\xi f_{i}) = \emptyset.
\end{align}
We now show that for all $m \geq 1$ and any collection of elements $\gamma_{1}, \dots, \gamma_{m} \in S$, we have
\begin{align}
\label{TreePf7}
d(\gamma_{m}^{-1} \cdots \gamma_{1}^{-1}, x) > 2t_{0}.
\end{align} By Lemma \ref{ConvGpAMS} and the fact that $S \subset B_{t_{0}}(x)$, we have
\begin{align}
\label{forobs1}
d(\gamma^{-1}, x) \geq d(\gamma^{-1}, \gamma) - d(\gamma, x) > \tilde{\epsilon} - t_{0} > 3t_{0} > 2t_{0},
\end{align} for any $\gamma \in S$. This proves the base case. Assume now that (\ref{TreePf7}) holds for $m \geq 1$ and let $\gamma_{m+1} \in S$. We want to show that $d(\gamma_{m+1}^{-1} \gamma_{m}^{-1} \cdots \gamma_{1}^{-1}, x) > 2t_{0}$. Set $\xi := \gamma_{m}^{-1} \cdots \gamma_{1}^{-1}$. By the inductive hypothesis,
\begin{align*}
d (\xi, \gamma_{m+1}) \geq d(\xi, x) - d(x, \gamma_{m+1}) > 2t_{0} - t_{0} = t_{0},
\end{align*} and therefore by (\ref{inclu2})
\begin{align*}
\gamma_{m+1}^{-1} \xi \in \gamma_{m+1}^{-1} \big( (\Gamma \sqcup M) \smallsetminus B_{t_{0}}(\gamma_{m+1}) \big) \subset B_{t_{0}} (\gamma_{m+1}^{-1}).
\end{align*} Along with the base case of the induction, this gives
\begin{align*}
d(\gamma_{m+1}^{-1} \cdots \gamma_{1}^{-1}, x) = d(\gamma_{m+1}^{-1} \xi, x) \geq d(x, \gamma_{m+1}^{-1}) - d(\gamma_{m+1}^{-1}, \gamma_{m+1}^{-1} \xi) > 3t_{0} - t_{0} = 2 t_{0},
\end{align*} which concludes the proof of (\ref{TreePf7}). 
\par 
We are now in the position to verify properties $(1)-(4)$ in the statement of the theorem. 
\par 
\begin{proof} [Proof of Property (1)]
Fix any $m \geq 0$ and $\gamma \in S^{m}$. Let $\eta \in \gamma \cdot S$. Using what we just proved, 
\begin{align}
\label{EstOnTree}
d(\gamma^{-1}, \gamma^{-1} \eta) \geq d(\gamma^{-1}, x) - d(x, \gamma^{-1} \eta) > 2t_{0} - t_{0} = t_{0}
\end{align}
Thus by Lemma \ref{TreeLem3},
\begin{align*}
\mathcal{S}_{\frac{t_{0}}{4}}(\eta) = \gamma \big( \gamma^{-1} \mathcal{S}_{\frac{t_{0}}{4}}(\gamma \cdot (\gamma^{-1} \eta)) \big) \subset \gamma (\gamma^{-1} \mathcal{S}_{\frac{t_{0}}{2}}(\gamma)) = S_{\frac{t_{0}}{2}}(\gamma),
\end{align*} which verifies part $(1)$. 
\end{proof}
\begin{proof} [Proof of Property (2)]
Let $\xi, \eta \in \gamma \cdot S$ be distinct children of $\gamma$. We want to show that 
\begin{align*}
\mathcal{S}_{\frac{t_{0}}{4}}(\xi) \cap \mathcal{S}_{\frac{t_{0}}{4}}(\eta) = \emptyset.
\end{align*}
Without loss of generality, assume that $||\gamma^{-1} \eta||_{\sigma} \leq ||\gamma^{-1} \xi||_{\sigma}$. By construction, $\gamma^{-1} \xi$ and $\gamma^{-1} \eta$ are both in $S$. Hence, 
\begin{align*}
\mathcal{S}_{\frac{t_{0}}{4}}(\eta) \cap \mathcal{S}_{\frac{t_{0}}{4}}(\xi) &= \mathcal{S}_{\frac{t_{0}}{4}}(\gamma \cdot \gamma^{-1} \eta) \cap \mathcal{S}_{\frac{t_{0}}{4}}(\gamma \cdot \gamma^{-1} \xi) \\
&= \gamma \cdot \Big( \gamma^{-1} \mathcal{S}_{\frac{t_{0}}{4}}(\gamma \cdot \gamma^{-1} \eta) \cap \gamma^{-1} \mathcal{S}_{\frac{t_{0}}{4}} (\gamma \cdot \gamma^{-1} \xi) \Big) \\
&\subset \gamma \cdot \Big( \mathcal{S}_{\frac{t_{0}}{8}}(\gamma^{-1} \eta) \cap \mathcal{S}_{\frac{t_{0}}{8}}(\gamma^{-1} \xi) \Big) = \emptyset,
\end{align*} where the inclusion in the third line follows from Lemma \ref{TreeLem3}, and where the last set being the empty set follows from Lemma \ref{TreeLem2}. This verifies part $(2)$. 
\end{proof}
\par 
\begin{proof} [Proof of Property (3)]
Define
\begin{align*}
D_{0} := \max_{\zeta \in S} ||\zeta^{-1}||_{\sigma}.
\end{align*} Given $\eta \in \gamma \cdot S$, we have $\gamma^{-1} \eta \in S$, hence
\begin{align*}
||\eta^{-1} \gamma||_{\sigma} = ||(\gamma^{-1} \eta)^{-1}||_{\sigma} \leq D_{0},
\end{align*} which proves part $(3)$. 
\end{proof} 
Finally, we prove part (4). 
\begin{proof} [Proof of Property (4)]
For any $\eta \in \gamma \cdot S$, (\ref{EstOnTree}) gives $d(\gamma^{-1}, \gamma^{-1} \eta) > t_{0}$. So by (\ref{TreePf1}), we have
\begin{align*}
e^{-||\eta||_{\sigma}} = e^{-||\gamma \cdot \gamma^{-1} \eta||_{\sigma}} \geq A e^{-||\gamma||_{\sigma}} e^{-||\gamma^{-1} \eta||_{\sigma}}.
\end{align*} Combining this with (\ref{TreePf3}), we conclude that
\begin{align*}
\sum_{\eta \in \gamma \cdot S} e^{-\delta ||\eta||_{\sigma}} \geq A^{\delta} e^{- \delta ||\gamma||_{\sigma}} \sum_{\eta \in \gamma \cdot S} e^{-\delta ||\gamma^{-1} \eta||_{\sigma}} = A^{\delta} e^{- \delta ||\gamma||_{\sigma}} \sum_{\xi \in S} e^{- \delta ||\xi||_{\sigma}} \geq e^{- \delta ||\gamma||_{\sigma}},
\end{align*} as desired. 
\end{proof}
Having established all the properties, it remains to check that $\mathcal{T}$ is a free semigroup. This means that if $\alpha \in \mathcal{T}$ is any element that can be written both as $\beta_{1} \cdots \beta_{k}$ and as $\gamma_{1} \cdots \gamma_{j}$ where $\beta_{i}, \gamma_{l} \in S$ for all $1 \leq i \leq k$ and $1 \leq l \leq j$, then we must have $k=j$ and $\beta_{i} = \gamma_{i}$ for all $1 \leq i \leq k$. There are two cases to consider:
\par 
$\textbf{Case 1.}$ In this case, we have $k=j$. Then there is a largest integer $0 \leq m \leq k$ so that $\beta_{i} = \gamma_{i}$ for all $0 \leq i \leq m$ (where we define $\beta_{0} = \gamma_{0} := \mathrm{id}$). If $m = k$, then there is nothing to show, so suppose $m < k$. Then we obtain
\begin{align}
\label{case1}
\beta_{m+1} \beta_{m+2} \cdots \beta_{k} = \gamma_{m+1} \gamma_{m+2} \cdots \gamma_{k},
\end{align} and $\beta_{m+1} \neq \gamma_{m+1}$. But then repeatedly applying property (1) (along with the fact that $\mathcal{S}_{\frac{t_{0}}{2}}(\alpha) \subset \mathcal{S}_{\frac{t_{0}}{4}}(\alpha)$ for any $\alpha \in \Gamma$) yields
\begin{align*}
\mathcal{S}_{\frac{t_{0}}{4}}(\beta_{m+1} \beta_{m+2} \cdots \beta_{k}) \subset \mathcal{S}_{\frac{t_{0}}{2}}(\beta_{m+1} \beta_{m+2} \cdots \beta_{k-1}) \subset \cdots \subset \mathcal{S}_{\frac{t_{0}}{2}} (\beta_{m+1}) \subset \mathcal{S}_{\frac{t_{0}}{4}}(\beta_{m+1})
\end{align*} and likewise
\begin{align*}
\mathcal{S}_{\frac{t_{0}}{4}}(\gamma_{m+1} \gamma_{m+2} \cdots \gamma_{k}) \subset \mathcal{S}_{\frac{t_{0}}{2}}(\gamma_{m+1}). 
\end{align*} But by property (2), we know that $\mathcal{S}_{\frac{t_{0}}{4}}(\beta_{m+1})$ and $\mathcal{S}_{\frac{t_{0}}{4}}(\gamma_{m+1})$ are disjoint, hence also $\mathcal{S}_{\frac{t_{0}}{4}}(\beta_{m+1} \beta_{m+2} \cdots \beta_{k})$ and $\mathcal{S}_{\frac{t_{0}}{4}}(\gamma_{m+1} \gamma_{m+2} \cdots \gamma_{k})$ are disjoint, which contradicts $(\ref{case1})$.
\par 
$\textbf{Case 2.}$ It remains to consider the case when $k \neq j$. Without loss of generality, suppose $k < j$. Let $m \geq 1$ be the smallest integer so that $\beta_{m} \neq \gamma_{m}$. If $m < k$, then we can argue just as in the previous case. If however $\beta_{i} = \gamma_{i}$ for $1 \leq i \leq k$, then we obtain
\begin{align*}
    \mathrm{id} = \gamma_{k+1} \cdots \gamma_{j}.
\end{align*} Again, property (1) allows us to conclude that 
\begin{align}
\label{case2}
\mathcal{S}_{\frac{t_{0}}{4}}(\mathrm{id}) = \mathcal{S}_{\frac{t_{0}}{4}}(\gamma_{k+1} \cdots \gamma_{j}) \subset \cdots \subset \mathcal{S}_{\frac{t_{0}}{4}}(\gamma_{k+1}).
\end{align} By property (2), we know that the elements of $\{\mathcal{S}_{\frac{t_{0}}{4}}(\eta)\}_{\eta \in S}$ are pairwise disjoint. Recalling that the set $S$ has at least two elements, property (1) then implies that the shadow $\mathcal{S}_{\frac{t_{0}}{4}}(\gamma_{k+1})$ is a proper subset of $\mathcal{S}_{\frac{t_{0}}{4}}(\mathrm{id})$. Hence $(\ref{case2})$ implies that $\mathcal{S}_{\frac{t_{0}}{4}}(\mathrm{id})$ is a $\textit{proper}$ subset of itself, which is impossible. This concludes the proof of the second case, and with it the proof of the proposition.
\end{proof}
\subsection{Properties of Bishop--Jones Semigroups}
\label{PropertiesSubsection}
In this section, we establish numerous useful properties of Bishop--Jones semigroups and prove properties (1), (2), and (5) of Theorem \ref{MainThm}. 
\par 
One immediate corollary of Proposition \ref{BJtree}, which establishes property (1) in Theorem \ref{MainThm}, is the following result. Fix $0 < \delta < \delta_{\sigma}(\Gamma)$ and let $\mathcal{T} = \mathcal{T}_{\delta}$ be the associated Bishop--Jones semigroup constructed in Proposition \ref{BJtree}.
\begin{cor}
\label{CritExpApprox}
The critical exponent of $\mathcal{T}$ satisfies
\begin{align*}
    \delta_{\sigma}(\mathcal{T}) \geq \delta.
\end{align*}
\end{cor}
\begin{proof}
We need to show that 
\begin{align*}
    \sum_{\gamma \in \mathcal{T}} e^{-\delta ||\gamma||_{\sigma}} = \infty.
\end{align*} Inductively applying property (4) of Proposition \ref{BJtree}, we see that 
\begin{align*}
    \sum_{\gamma \in S^{m}} e^{-\delta ||\gamma||_{\sigma}} \geq e^{-\delta ||\mathrm{id}||_{\sigma}} > 0.
\end{align*} It follows that 
\begin{align*}
    \sum_{\gamma \in \mathcal{T}} e^{-\delta ||\gamma||_{\sigma}} = \sum_{m=0}^{\infty} \sum_{\gamma \in S^{m}} e^{-\delta ||\gamma||_{\sigma}} \geq \sum_{m=0}^{\infty} e^{-\delta ||\mathrm{id}||_{\sigma}} = \infty,
\end{align*} which concludes the proof. 
\end{proof}
The following technical lemma will be very useful when attempting to control the asymptotic behavior of $\mathcal{T}$ and of $\mathcal{T}^{-1} := \{ \gamma^{-1} : \gamma \in \mathcal{T}\}$ (see in particular Lemma \ref{TreeMag1} and Lemma \ref{propersubset}). Let $\epsilon_{0}$ be as in item (ii) in the proof of Proposition \ref{BJtree}.
\begin{lem}
\label{obs1}
For every $m \geq 1$ and $\gamma \in S$, we have
\begin{align*}
S^{m} \subset (\Gamma \sqcup M) \smallsetminus B_{\epsilon_{0}}(\gamma^{-1}).
\end{align*}
\end{lem}
\begin{proof}
We prove the claim by induction. Let $\alpha, \beta \in S$ be arbitrary. By (\ref{forobs1}), we have $d(\alpha^{-1},x) > 3t_{0}$. By construction, we have $S \subset B_{t_{0}}(x)$. Hence,
\begin{align*}
d(\alpha^{-1}, \beta) \geq d(\alpha^{-1}, x) - d(x, \beta) > 3t_{0} - t_{0} = 2t_{0} > \epsilon_{0}.
\end{align*} Since $\alpha$ and $\beta$ were arbitrary, this proves the claim for $m=1$. Suppose now that $S^{m} \subset (\Gamma \sqcup M) \smallsetminus B_{\epsilon_{0}}(\gamma^{-1})$ for all $\gamma \in S$. Let $\alpha \in S^{m+1}$ be arbitrary. Then $\alpha = \eta \beta$ for some $\eta \in S$ and $\beta \in S^{m}$. By the inductive hypothesis, $\beta \in S^{m} \subset (\Gamma \sqcup M) \smallsetminus B_{\epsilon_{0}}(\eta^{-1})$, so by (\ref{inclu1}) we obtain
\begin{align*}
\alpha = \eta \beta \in \eta \big( (\Gamma \sqcup M) \smallsetminus B_{\epsilon_{0}}(\eta^{-1}) \big) \subset B_{\epsilon_{0}}(\eta) \subset B_{t_{0} + \epsilon_{0}}(x).
\end{align*} Thus for any $\gamma \in S$ we have
\begin{align*}
d(\gamma^{-1}, \alpha) \geq d(\gamma^{-1},x) - d(x,\alpha) \geq 3t_{0} - t_{0} - \epsilon_{0} > \epsilon_{0}.
\end{align*} Since $\alpha \in S^{m+1}$ was arbitrary, this shows that 
\begin{align*}
    S^{m+1} \subset (\Gamma \sqcup M) \smallsetminus B_{\epsilon_{0}}(\gamma^{-1})
\end{align*} for all $\gamma \in S$, as desired.
\end{proof}
In the following lemma, we establish property (2) of Theorem \ref{MainThm}. 
\begin{lem}
\label{TreeMag1}
Given any expanding coarse-cocycle $\phi : \Gamma \times M \rightarrow \mathbb{R}$, there exist constants $B_{1} = B_{1}(\phi) > 1$ and $b_{1} = b_{1}(\phi) > 0$ so that
\begin{align*}
   \frac{1}{B_{1}} |\gamma|_{S} - b_{1} \leq ||\gamma||_{\phi} \leq B_{1} |\gamma|_{S} + b_{1}
\end{align*} for all $\gamma \in \mathcal{T}$. 
\end{lem}
\begin{proof}
By Lemma \ref{obs1}, we know that for any $\gamma \in S$, $n \geq 1$, and $\alpha \in S^{n}$, we have $d(\gamma^{-1}, \alpha) \geq \epsilon_{0}$. Let $C_{\phi} = C_{\phi}(\epsilon_{0}) > 0$ be the constant provided by part (3) of Proposition \ref{CocycleProperties} associated to the constant $\epsilon_{0}$ and the coarse-cocycle $\phi$. Fix $m$ sufficiently large so that 
\begin{align*}
    ||\beta||_{\phi} \geq C_{\phi} + 1
\end{align*} for all $\beta \in S^{m}$. For any $n \geq 0$, we can write $\gamma \in S^{n}$ in the form 
\begin{align*}
    \gamma = \gamma_{1} \cdots \gamma_{q} \alpha,
\end{align*} where $n = qm + r$, $q = \left \lfloor{n/m}\right \rfloor $, $0 \leq r \leq m-1$, $\gamma_{j} \in S^{m}$, and $\alpha \in S^{r}$. Using Lemma \ref{obs1}, a calculation of the same sort as in the proof of Proposition \ref{BJtree} shows that 
\begin{align*}
d(\gamma_{i}^{-1}, \gamma_{i+1} \cdots \gamma_{q}) \geq t_{0} - \epsilon_{0} > \epsilon_{0}
\end{align*} for all $1 \leq i \leq q-1$. Let $C' > 0$ be the constant provided by part (1) of Proposition~\ref{CocycleProperties} associated to the finite subset $\bigcup_{i=0}^{m-1} S^{i}$ of $\Gamma$ and the coarse-cocycle $\phi$. Combining all of the above facts, we obtain
\begin{align*}
    ||\gamma||_{\phi} &= ||\gamma_{1} \cdots \gamma_{q} \alpha||_{\phi} \\
    &\geq ||\gamma_{1} \cdots \gamma_{q}||_{\phi} - C' \\
    &\geq ||\gamma_{1}||_{\phi} + ||\gamma_{2} \cdots \gamma_{q}||_{\phi} - C_{\phi} - C' \\
    &\geq \cdots \geq ||\gamma_{1}||_{\phi} + \cdots + ||\gamma_{q}||_{\phi} - ((q-1)C_{\phi} + C') \\
    &\geq q(C_{\phi} + 1) - ((q-1)C_{\phi} + C') \\
    &= q + C_{\phi} - C' \geq q - C' \\
    &= \frac{n}{m} - \Big( \frac{r}{m} + C' \Big) \\
    & \geq \frac{1}{m} |\gamma|_{S} - (C'+1).
\end{align*} Now let $C'' = C''(S)$ be the constant provided by part (1) of Proposition \ref{CocycleProperties} corresponding to the finite set $S$ and coarse-cocycle $\phi$. Since $\gamma = \xi_{1} \cdots \xi_{n}$ for some $\xi_{i} \in S$, we obtain
\begin{align*}
    ||\gamma||_{\phi} &= ||\xi_{1} \cdots \xi_{n}||_{\phi} \\
    &\leq ||\xi_{1} \cdots \xi_{n-1}||_{\phi} + C'' \\
    &\leq \cdots \leq ||\mathrm{id}||_{\phi} + nC'' \\
    &= C''|\gamma|_{S} + ||\mathrm{id}||_{\phi}.
\end{align*} Thus the lemma holds with $B_{1} = \max \{m, C''\}$ and $b_{1} = \max \{||\mathrm{id}||_{\phi}, C'+1\}$.
\end{proof}
As we remarked earlier, property (5) in Theorem \ref{MainThm} is an immediate consequence of this lemma. 
\par 
We conclude this section by examining the relation between Bishop--Jones semigroups and the conical limit set of the ambient group. This will be essential in the coming results. The next lemma shows that any two elements of the Bishop--Jones semigroup which differ starting from the first generation are a definite distance apart.
\begin{lem}
\label{DefiniteDistance}
There exists $r > 0$ so that for any two elements $\gamma_{1} = \alpha_{1} \cdots \alpha_{k} \in S^{k}$ and $\gamma_{2} = \beta_{1} \cdots \beta_{j} \in S^{j}$ of the Bishop--Jones semigroup with $\alpha_{1} \neq \beta_{1}$, we have $d(\gamma_{1}, \gamma_{2}) \geq r$.
\end{lem}
\begin{proof}
Suppose not. Then there exist sequences $\{\gamma_{n}\}$ and $\{\eta_{n}\}$ of elements in $\mathcal{T}$ with $\gamma_{n} = \alpha_{1,n} \cdots, \alpha_{k(n),n}$ and $\eta_{n} = \beta_{1,n} \cdots \beta_{j(n),n}$ so that $\alpha_{1,n} \neq \beta_{1,n}$ for each $n$, but $\lim_{n \to \infty} d(\gamma_{n}, \eta_{n}) = 0$. After passing to sufficiently many subsequences, we may assume without loss of generality that: 
\begin{itemize}
    \item[(1)] there are elements $\alpha_{1} \neq \beta_{1}$ of $S$ so that $\alpha_{1,n} = \alpha_{1}$ and $\beta_{1,n} = \beta_{1}$ for all $n$, and
    \item[(2)] $\lim_{n \to \infty} \gamma_{n} = \lim_{n \to \infty} \eta_{n}$. 
\end{itemize} By property (2) of Proposition $\ref{BJtree}$, we have $\mathcal{S}_{\frac{t_{0}}{4}} (\alpha_{1}) \cap \mathcal{S}_{\frac{t_{0}}{4}}(\beta_{1}) = \emptyset$. Repeatedly applying property (1) of Proposition $\ref{BJtree}$ yields
\begin{align*}
\mathcal{S}_{\frac{t_{0}}{4}}(\gamma_{n}) \subset \mathcal{S}_{\frac{t_{0}}{2}}(\alpha_{1}) \subset \mathcal{S}_{\frac{t_{0}}{4}}(\alpha_{1}) \ \ \mathrm{and} \ \ \mathcal{S}_{\frac{t_{0}}{4}}(\eta_{n}) \subset \mathcal{S}_{\frac{t_{0}}{2}}(\beta_{1}) \subset \mathcal{S}_{\frac{t_{0}}{4}}(\beta_{1}),
\end{align*} for all $n$. But as $\{\gamma_{n}\}, \{\eta_{n}\} \subset \Gamma$ are escaping sequences, part (2) of Proposition \ref{ShadowProperties} implies that
\begin{align*}
    &\lim_{n \to \infty} \Big( \inf_{q \in \mathcal{S}_{\frac{t_{0}}{4}}(\alpha_{1})} d(\gamma_{n},q) \Big) \leq \lim_{n \to \infty} \Big( \inf_{q \in \mathcal{S}_{\frac{t_{0}}{4}}(\gamma_{n})} d(\gamma_{n},q) \Big) = 0, \ \ \mathrm{and} \\
    &\lim_{n \to \infty} \Big( \inf_{s \in \mathcal{S}_{\frac{t_{0}}{4}}(\beta_{1})} d(\eta_{n},s) \Big) \leq \lim_{n \to \infty} \Big( \inf_{s \in \mathcal{S}_{\frac{t_{0}}{4}}(\eta_{n})} d(\eta_{n},s) \Big) = 0.
\end{align*} Since $\mathcal{S}_{\frac{t_{0}}{4}}(\alpha_{1})$ and $\mathcal{S}_{\frac{t_{0}}{4}}(\beta_{1})$ are closed subsets of $M$, we have $\lim_{n \to \infty} \gamma_{n} \in \mathcal{S}_{\frac{t_{0}}{4}}(\alpha_{1})$ and $\lim_{n \to \infty} \eta_{n} \in \mathcal{S}_{\frac{t_{0}}{4}}(\beta_{1})$. But $\lim_{n \to \infty} \gamma_{n} = \lim_{n \to \infty} \eta_{n}$, so 
\begin{align*}
\mathcal{S}_{\frac{t_{0}}{4}}(\alpha_{1}) \cap \mathcal{S}_{\frac{t_{0}}{4}}(\beta_{1}) \neq \emptyset,
\end{align*} which is a contradiction. This concludes the proof.
\end{proof}
The next lemma shows that any geodesic sequence in $\mathcal{T}$ limits to a point in $M$ which is $\frac{t_{0}}{2}$-uniformly conical.
\begin{lem}
\label{GeodesicConicalPoint}
Let $\{\gamma_{n}\} \subset \mathcal{T}$ be a geodesic sequence in a Bishop--Jones semigroup $\mathcal{T}$, that is, $\gamma_{n+1} \in \gamma_{n} \cdot S$ for all $n \geq 1$. Then $\lim_{n \to \infty} \gamma_{n}$ exists. Moreover, denoting $w \in M$ by this limit, we have $w \in \Lambda_{\frac{t_{0}}{2}}^{\mathrm{con}}(\Gamma)$.
\end{lem}
\begin{proof}
By repeatedly applying property (1) in Proposition \ref{BJtree}, we obtain the following decreasing nested sequence of shadows:
\begin{align*}
    \mathcal{S}_{\frac{t_{0}}{2}}(\gamma_{1}) &\supset \mathcal{S}_{\frac{t_{0}}{2}}(\gamma_{2}) \supset \cdots \supset \mathcal{S}_{\frac{t_{0}}{2}}(\gamma_{n}) \supset \mathcal{S}_{\frac{t_{0}}{2}}(\gamma_{n+1}) \supset \cdots.
\end{align*} We claim that the intersection
\begin{align}
\label{Conical1}
\bigcap_{n=1}^{\infty} \mathcal{S}_{\frac{t_{0}}{2}}(\gamma_{n}).
\end{align} is a singleton. Since each shadow $\mathcal{S}_{\frac{t_{0}}{2}}(\gamma_{n})$ is closed and compact, their intersection is nonempty. But by part (2) of Proposition \ref{ShadowProperties}, $\lim_{n \to \infty} \mathrm{diam} \ \mathcal{S}_{\frac{t_{0}}{2}}(\gamma_{n}) = 0$, hence $\bigcap_{n=1}^{\infty} \mathcal{S}_{\frac{t_{0}}{2}}(\gamma_{n})$ is a singleton, which we denote by $\{w\} \subset M$. Again by part (2) of Proposition \ref{ShadowProperties}, there exist $x_{n} \in \mathcal{S}_{\frac{t_{0}}{2}}(\gamma_{n})$ so that $d(\gamma_{n}, x_{n}) \rightarrow 0$ as $n \rightarrow \infty$. Notice also that $d(x_{n}, w) \leq \mathrm{diam} \ \mathcal{S}_{\frac{t_{0}}{2}}(\gamma_{n}) \rightarrow 0$ as $n \rightarrow \infty$. Therefore,
\begin{align*}
    d(\gamma_{n}, w) \leq d(\gamma_{n}, x_{n}) + d(x_{n}, w) \rightarrow 0 \ \ \mathrm{as} \ \ n \rightarrow \infty,
\end{align*} which shows that $\lim_{n \to \infty} \gamma_{n} = w$. We conclude by Lemma \ref{ShadowsAndLimitSets} that $w \in \Lambda_{\frac{t_{0}}{2}}^{\mathrm{con}}(\Gamma)$.
\end{proof} 
Let $E = E_{\delta} \subset \Lambda(\Gamma)$ denote the set of accumulation points of $\mathcal{T}$. Since $\mathcal{T}$ is infinite, $E$ is nonempty. We now show that $E$ is uniformly conical.
\begin{lem}
\label{UniformlyConical}
We have $E \subset \Lambda_{\frac{t_{0}}{2}}^{\mathrm{con}}(\Gamma)$.
\end{lem}
\begin{proof}
Fix $u \in E$ and let $\{\gamma_{n}\}$ be a distinct sequence of elements of $\mathcal{T}$ so that $\gamma_{n} \rightarrow u$. By a diagonal extraction argument, let us find a subsequence $\{\eta_{n}\}$ of $\{\gamma_{n}\}$ so that for all $k \leq n$, $\eta_{k}$ and $\eta_{n}$ share the same first, second, ..., and $k$-th ancestors; that is, they are descendents of the same elements in $S$, $S^{2}$, all the way up to $S^{k}$. For clarity, we describe how we construct this subsequence. 
\par 
In what follows, for each $k \geq 1$, $\phi_{k} : \mathbb{N} \rightarrow \mathbb{N}$ denotes a strictly increasing function. We first pass to a subsequence $\{\gamma_{\phi_{1}(1)}, \gamma_{\phi_{1}(2)}, \gamma_{\phi_{1}(3)}, \dots \}$ of $\{\gamma
_{n}\}$ so that all elements of this subsequence descend from the same element of $S$ (or perhaps an element of this subsequence is this element of $S$). Next we pass to a subsequence $\{\gamma_{\phi_{2} \circ \phi_{1}(1)}, \gamma_{\phi_{2} \circ \phi_{1}(2)}, \gamma_{\phi_{2} \circ \phi_{1}(3)}, \dots \}$ of $\{\gamma_{\phi_{1}(n)}\}_{n \geq 1}$ so that all elements of this subsequence descend from the same element of $S^2$ (or perhaps an element of this subsequence is this element of $S^2$). Continue passing to such subsequences and then define
\begin{align*}
    \eta_{n} := \gamma_{\phi_{n} \circ \phi_{n-1} \circ \cdots \circ \phi_{1}(n)}.
\end{align*} Then the sequence $\{\eta_{n}\}$ has the aforementioned property. Now consider the sequence $\{\alpha_{n}\} \subset \mathcal{T}$, where $\alpha_{n} \in S^{n}$ is the ancestor of $\eta_{n}$ in $S^{n}$. By construction, $\{\alpha_{n}\}$ is a geodesic sequence. By Lemma \ref{GeodesicConicalPoint}, $w := \lim_{n \to \infty} \alpha_{n}$ exists and moreover $w \in \Lambda_{\frac{t_{0}}{2}}^{\mathrm{con}}(\Gamma)$. It therefore suffices to show that $w = u$. Recall that $\eta_{n} \rightarrow u$ and, by property (1) of Proposition \ref{BJtree}, that $\mathcal{S}_{\frac{t_{0}}{4}}(\eta_{n}) \subset \mathcal{S}_{\frac{t_{0}}{2}}(\alpha_{n})$. Part (2) of Proposition \ref{ShadowProperties} then yields
\begin{align*}
    &\lim_{n \to \infty} \inf_{z \in \mathcal{S}_{\frac{t_{0}}{2}}(\alpha_{n})} d(u,z) \leq  \lim_{n \to \infty} \inf_{z \in \mathcal{S}_{\frac{t_{0}}{4}}(\eta_{n})} d(u,z) = 0, \\
     &\lim_{n \to \infty} \inf_{z \in \mathcal{S}_{\frac{t_{0}}{2}}(\alpha_{n})} d(w,z) = 0, \ \ \mathrm{and} \\
     &\lim_{n \to \infty} \mathrm{diam} \ \mathcal{S}_{\frac{t_{0}}{2}} (\alpha_{n}) = 0.
\end{align*} Therefore there are $u_{n}, w_{n} \in \mathcal{S}_{\frac{t_{0}}{2}}(\alpha_{n})$ so that $d(u,u_{n}) \rightarrow 0$, $d(w,w_{n}) \rightarrow 0$, and $d(u_{n}, w_{n}) \rightarrow 0$. It follows that 
\begin{align*}
    d(u,w) \leq d(u,u_{n}) + d(u_{n}, w_{n}) + d(w_{n}, w) \rightarrow 0,
\end{align*} which shows that $u = w$. This concludes the proof.
\end{proof}
\subsection{Critical Exponent Gap for Bishop--Jones Semigroups}
\label{ExpGapSubsection}
In this section, we prove properties (3) and (4) of Theorem \ref{MainThm}. We begin by establishing property (3).
\begin{prop}
\label{LogBounds}
With $B_{1} = B_{1}(\sigma) > 1$ and $b_{1} = b_{1}(\sigma) > 0$ as in part (2) of Theorem \ref{MainThm}, we have  
\begin{align*}
        \frac{1}{B_{1}} \log (\# S) \leq \delta_{\sigma}(\mathcal{T}) \leq B_{1} \log (\# S).
    \end{align*} In particular, $\delta_{\sigma}(\mathcal{T}) < \infty$.
\end{prop}
\begin{proof}
Consider the Poincar\'e series of $\mathcal{T}$ with respect to the word-length $|\cdot|_{S}$, that is the series
\begin{align*}
    Q_{S}^{\mathcal{T}}(t) := \sum_{\gamma \in \mathcal{T}} e^{-t|\gamma|_{S}},
\end{align*} and let 
\begin{align*}
\delta_{S}(\mathcal{T}) &:= \inf \{t > 0 : Q_{S}^{\mathcal{T}}(t) < \infty \} = \limsup_{k \to \infty} \frac{1}{k} \log \{ \gamma \in \mathcal{T} : |\gamma|_{S} \leq k \}
\end{align*}
denote its critical exponent.  Recall from the proof of Proposition \ref{BJtree} that $\#S \geq 2$. Since $\mathcal{T}$ is a free semigroup, the number of elements of $\mathcal{T}$ of word length $k \geq 0$ is $(\# S)^k$. Thus we obtain
\begin{align*}
\delta_{S}(\mathcal{T}) &= \limsup_{k \to \infty} \frac{1}{k} \log \big(\# \{ \gamma \in \mathcal{T} : |\gamma|_{S} \leq k \} \big) \\
&= \limsup_{k \to \infty} \frac{1}{k} \Big( \log \sum_{i=0}^{k} (\# S)^{i} \Big) \\
&= \limsup_{k \to \infty} \frac{1}{k} \log \bigg( \frac{(\# S)^{k+1} - 1}{\# S - 1} \bigg) \\
&= \limsup_{k \to \infty} \frac{1}{k} \log \big( (\# S)^{k+1} - 1 \big) - \frac{1}{k} \log (\#S - 1) \\
&= \log (\#S).
\end{align*}
It thus suffices to show that 
\begin{align}
\label{LogBounds1}
\frac{1}{B_{1}} \delta_{S}(\mathcal{T}) \leq \delta_{\sigma}(\mathcal{T}) \leq B_{1} \delta_{S}(\mathcal{T}). 
\end{align} Fix an arbitrary $t > \delta_{S}(\mathcal{T})$. Then, $Q_{S}^{\mathcal{T}}(t) < \infty$. We claim that $Q_{\sigma}^{\mathcal{T}}(B_{1}t) < \infty$. By part (2) of Theorem \ref{MainThm}
\begin{align*}
    -B_{1}t||\gamma||_{\sigma} \leq -B_{1}t \bigg( \frac{1}{B_{1}}|\gamma|_{S} - b_{1} \bigg) = b_{1}B_{1}t - t |\gamma|_{S}
\end{align*} for all $\gamma \in \mathcal{T}$, hence
\begin{align*}
    Q_{\sigma}^{\mathcal{T}}(B_{1}t) = \sum_{\gamma \in \mathcal{T}} e^{-B_{1} t ||\gamma||_{\sigma}} \leq e^{b_{1}B_{1}t} Q_{S}^{\mathcal{T}}(t) < \infty,
\end{align*} as desired. We conclude that $\delta_{\sigma}(\mathcal{T}) \leq B_{1} \delta_{S}(\mathcal{T})$, and the other inequality in (\ref{LogBounds1}) is proven entirely analogously. This concludes the proof.
\end{proof}
If $\delta_{\sigma}(\Gamma) = \infty$, then the above result gives $\delta_{\sigma}(\mathcal{T}) < \delta_{\sigma}(\Gamma)$. Hence to conclude the proof of property (4) of Theorem \ref{MainThm} (and thus the proof of the entire theorem), it remains to consider the case when $\delta_{\sigma}(\Gamma) < \infty$. This is precisely the content of the following result.
\begin{thm}
\label{SemigroupExpGap}
Suppose $\Gamma \subset \mathrm{Homeo}(M)$ is a convergence group, $\sigma$ is an expanding coarse-cocycle, and $\delta_{\sigma}(\Gamma) < \infty$. Fix $0 < \delta < \delta_{\sigma}(\Gamma)$ and let $\mathcal{T} = \mathcal{T}_{\delta}$ be the corresponding Bishop--Jones subsemigroup of $\Gamma$. Then,
\begin{align*}
    \delta_{\sigma}(\mathcal{T}) < \delta_{\sigma}(\Gamma).
\end{align*}
\end{thm} 
\begin{rem}
The proof of this theorem will be in the spirit of classical work of Dal'bo--Otal--Peign\'e (see Proposition 2 of \cite{DOP}, and, for additional examples of this technique, Theorem 4.1 of \cite{CZZ2} and Theorem 4.3 of \cite{BCZZ1}). When one proves such critical exponent gap results for subgroups, the following hypotheses are typically needed: the first is that the Poincar\'e series of the subgroup diverges at the critical exponent (in which case, the subgroup is said to be of \emph{divergence type}), and the second is that the limit set of the subgroup is a proper subset of the limit set of the ambient group. As a step towards proving this theorem, we will show that the $\sigma$-Poincar\'e series of $\mathcal{T}$ indeed diverges at $\delta_{\sigma}(\mathcal{T})$ (that is, $\mathcal{T}$ is a semigroup of divergence type). Moreover, denoting by $E$ the set of accumulation points of $\mathcal{T}$ and $E'$ the set of accumulation points of $\mathcal{T}^{-1} := \{ \gamma^{-1} : \gamma \in \mathcal{T} \}$, we show that $E \cup E'$ is a proper subset of $\Lambda(\Gamma)$. Thus no such assumptions are required in the statement of Theorem $\ref{SemigroupExpGap}$. Establishing these facts is the crux of the proof of the theorem. 
\end{rem}
We begin by first showing that $\mathcal{T}$ is of divergence type. The main idea is the following. For $R > 0$, define 
\begin{align*}
    n(R) := \# \{\gamma \in \mathcal{T} : ||\gamma||_{\sigma} \leq R \}.
\end{align*} Since the set $E$ is uniformly conical, we can adapt an argument of Coornaert (see Theorem 7.2 of \cite{Coo}) to show that there exists a constant $C \geq 1$ so that 
\begin{align}
\label{asymptotic}
    n(R) \geq \frac{1}{C} e^{\delta_{\sigma}(\mathcal{T})R}
\end{align} for all $R \geq ||\mathrm{id}||_{\sigma}$. We are then able to show that the Poincar\'e series 
\begin{align*}
    Q_{\sigma}^{\mathcal{T}}(s) := \sum_{\gamma \in \mathcal{T}} e^{-s ||\gamma||_{\sigma}}
\end{align*} for $\mathcal{T}$ diverges at the critical exponent $s = \delta_{\sigma}(\mathcal{T})$, which allows us to modify classical arguments to prove Theorem \ref{SemigroupExpGap}. 
\par 
Before we can adapt Coornaert's argument, we need to address the following technical issue. In Coornaert's proof, the Patterson--Sullivan measure of his subgroup is supported on its limit set. However in our case, a coarse-Patterson--Sullivan measure for $\Gamma$ might not even assign positive measure to the set $E \subset \Lambda(\Gamma)$. Hence, we need to construct a measure $\nu$, supported on $E$, with some of the properties of a Patterson--Sullivan measure, in the sense which we now describe. 
\par 
We obtain such a measure by performing a similar construction to that of Theorem $4.1$ of \cite{BCZZ1}, which is an adaptation of the standard construction due to Patterson \cite{P} and Sullivan \cite{S}. Let $\delta := \delta_{\sigma}(\mathcal{T}) < \infty$ and for $x \in \Gamma \sqcup M$ let $\mathcal{D}_{x}$ denote the Dirac measure at $x$. By Lemma 3.1 in \cite{P}, there exists a non-decreasing function $\chi : \mathbb{R} \rightarrow \mathbb{R}_{\geq 1}$ so that 
\begin{itemize}
    \item[(a)] For every $\epsilon > 0$, there exists $R > 0$ such that $\chi(r+t) \leq e^{\epsilon t} \chi(r)$ for any $r \geq R$ and $t \geq 0$,
    \item[(b)] $\sum_{g \in \mathcal{T}} \chi(||g||_{\sigma}) e^{-\delta ||g||_{\sigma}} = \infty$. 
\end{itemize} (when $\sum_{g \in \mathcal{T}} e^{-\delta ||g||_{\sigma}} = \infty$, we may take $\chi \equiv 1$). 
\par 
For $s > \delta$, we define a Borel probability measure on $\Gamma \sqcup M$ by 
\begin{align*}
    \nu_{s} := \frac{1}{Q_{\sigma}^{\chi}(s)} \sum_{g \in \mathcal{T}} \chi(||g||_{\sigma}) e^{-s||g||_{\sigma}} \mathcal{D}_{g},
\end{align*} where $Q_{\sigma}^{\chi}(s) := \sum_{g \in \mathcal{T}} \chi(||g||_{\sigma}) e^{-s ||g||_{\sigma}}$, which is finite by property (a) above. Now fix $s_{n} \searrow \delta$ so that $\nu_{s_{n}} \overset{\ast}{\rightharpoonup} \nu$. 
Notice that $\nu$ is a Borel probability measure supported on $E$. The following lemma provides an upper bound on the $\nu$-measure of certain shadows of elements of the Bishop--Jones semigroup $\mathcal{T}$.
\begin{lem}
\label{ShadowUpperBound}
There exists a constant $C_{1} > 1$ so that 
\begin{align*}
    \nu(\mathcal{S}_{\frac{t_{0}}{2}}(\gamma)) \leq C_{1} e^{-\delta ||\gamma||_{\sigma}}
\end{align*} for all $\gamma \in \mathcal{T}$.
\end{lem}
\begin{proof}
Fix any $\gamma \in S^{n} \subset \mathcal{T}$. Recall from part (2) of Proposition \ref{BJtree} that if $\eta \in S^{n} \smallsetminus \{\gamma\}$, then $\mathcal{S}_{\frac{t_{0}}{2}}(\gamma) \cap \mathcal{S}_{\frac{t_{0}}{2}}(\eta) = \emptyset$. Thus there exists an open set $U \subset \Gamma \sqcup M$ so that $\mathcal{S}_{\frac{t_{0}}{2}}(\gamma) \subset U$ and $\overline{U} \cap \mathcal{S}_{\frac{t_{0}}{2}}(\eta) = \emptyset$ for all $\eta \in S^{n} \smallsetminus \{\gamma\}$. Notice that 
\begin{align}
\label{finite}
\# \{g \in \mathcal{T} \cap U : g \notin \gamma \mathcal{T} \} < \infty.
\end{align} Indeed, if not, then after passing to a subsequence, there exists a sequence of distinct elements $\{g_{n}\} \subset \mathcal{T} \cap U$, all descending from the same element $\eta \in S^{n} \smallsetminus \{\gamma\}$, and such that $g_{n} \rightarrow x \in \overline{U}$. By part (2) of Proposition \ref{ShadowProperties}, there exist $x_{n} \in \mathcal{S}_{\frac{t_{0}}{2}}(g_{n})$ so that $x_{n} \rightarrow x$. Repeatedly applying part (1) of Proposition \ref{BJtree} gives $\mathcal{S}_{\frac{t_{0}}{2}}(g_{n}) \subset \mathcal{S}_{\frac{t_{0}}{2}}(\eta)$, and since $\mathcal{S}_{\frac{t_{0}}{2}}(\eta)$ is closed this implies that $x \in \overline{U} \cap \mathcal{S}_{\frac{t_{0}}{2}}(\eta)$, which is a contradiction. Thus (\ref{finite}) holds.
\par 
We now will collect various estimates which will be needed to give an upper bound on the $\nu$-measure of $\mathcal{S}_{\frac{t_{0}}{2}}(\gamma)$.
\par 
As in Lemma \ref{TreeMag1}, we have $d(\gamma^{-1}, h) \geq \epsilon_{0}$ for all $h \in \mathcal{T}$ (where $\epsilon_{0} > 0$ is defined in the start of the proof of Proposition \ref{BJtree}). By part (3) of Proposition \ref{CocycleProperties}, there is a constant $C = C(\epsilon_{0}) > 0$ so that 
\begin{align}
\label{PoincIn1}
e^{-s_{n}||\gamma h||_{\sigma}} \leq e^{\delta C} \cdot e^{-\delta ||\gamma||_{\sigma}} \cdot e^{-s_{n} ||h||_{\sigma}}
\end{align} for all $h \in \mathcal{T}$ and all $s_{n}$.
\par 
Now fix $\epsilon > 0$ and let $R = R(\epsilon) > 0$ be as in the first property of the function $\chi$ defined above. By part (1) of Proposition \ref{CocycleProperties}, there is a constant $C' = C'(\gamma) > 0$ so that 
\begin{align*}
    ||h||_{\sigma} - C' \leq ||\gamma h||_{\sigma} \leq ||h||_{\sigma} + C'
\end{align*} for all $h \in \mathcal{T}$. Defining $F_{\epsilon}$ to be the finite set $F_{\epsilon} := \{ h \in \mathcal{T} : ||h||_{\sigma} < R(\epsilon)\}$, we then obtain
\begin{align}
\chi(||\gamma h||_{\sigma}) \leq \chi (||h||_{\sigma} + C') \leq e^{\epsilon C'} \chi(||h||_{\sigma})
\end{align} for all $h \in \mathcal{T} \smallsetminus F_{\epsilon}$. Set 
\begin{align*}
T_{\epsilon} := \{ g \in \mathcal{T} \cap U : g \notin \gamma (\mathcal{T} \smallsetminus F_{\epsilon}) \},
\end{align*} and observe from (\ref{finite}) and the fact that $F_{\epsilon}$ is finite that this set $T_{\epsilon}$ is also finite. We are now in a position to obtain an upper bound on $\nu(\mathcal{S}_{\frac{t_{0}}{2}}(\gamma))$.
\par 
By Urysohn's Lemma, there exists a continuous function $f : \Gamma \sqcup M \rightarrow [0,1]$ so that $f \equiv 1$ on $\mathcal{S}_{\frac{t_{0}}{2}}(\gamma)$ and $f \equiv 0$ on $(\Gamma \sqcup M) \smallsetminus U$. Using the above established facts, we obtain
\begin{align*}
\nu(\mathcal{S}_{\frac{t_{0}}{2}}(\gamma)) &\leq \int f \ d\nu = \lim_{s_{n} \searrow \delta} \int f \ d \nu_{s_{n}} \\
&= \lim_{s_{n} \searrow \delta} \frac{1}{Q_{\sigma}^{\chi}(s_{n})} \sum_{g \in \mathcal{T} \cap U} \chi (||g||_{\sigma}) e^{-s_{n}||g||_{\sigma}} f(g) \\
&\leq \lim_{s_{n} \searrow \delta} \frac{1}{Q_{\sigma}^{\chi}(s_{n})} \sum_{g \in \mathcal{T} \cap U} \chi (||g||_{\sigma}) e^{-s_{n}||g||_{\sigma}} \\
&\leq \lim_{s_{n} \searrow \delta} \frac{1}{Q_{\sigma}^{\chi}(s_{n})} \bigg( \sum_{g \in T_{\epsilon}} \chi (||g||_{\sigma}) e^{-s_{n}||g||_{\sigma}} + \sum_{h \in \mathcal{T} \smallsetminus F_{\epsilon}} \chi (||\gamma h||_{\sigma}) e^{-s_{n}||\gamma h||_{\sigma}}     \bigg) \\
&\leq \lim_{s_{n} \searrow \delta} \frac{1}{Q_{\sigma}^{\chi}(s_{n})} \bigg( \sum_{g \in T_{\epsilon}} \chi (||g||_{\sigma}) e^{-\delta ||g||_{\sigma}} + \sum_{h \in \mathcal{T} \smallsetminus F_{\epsilon}} \chi (||\gamma h||_{\sigma}) e^{-s_{n}||\gamma h||_{\sigma}} \bigg) \\
&= \lim_{s_{n} \searrow \delta} \frac{1}{Q_{\sigma}^{\chi}(s_{n})} \sum_{h \in \mathcal{T} \smallsetminus F_{\epsilon}} \chi (||\gamma h||_{\sigma}) e^{-s_{n}||\gamma h||_{\sigma}} \\
&\leq \lim_{s_{n} \searrow \delta} \frac{1}{Q_{\sigma}^{\chi}(s_{n})} e^{\epsilon C'} \cdot e^{\delta C} \cdot e^{-\delta ||\gamma||_{\sigma}} \sum_{h \in \mathcal{T} \smallsetminus F_{\epsilon}} \chi(||h||_{\sigma}) e^{-s_{n} ||h||_{\sigma}} \\
&\leq \lim_{s_{n} \searrow \delta} e^{\epsilon C'} \cdot e^{\delta C} \cdot e^{-\delta ||\gamma||_{\sigma}} \cdot \frac{1}{Q_{\sigma}^{\chi}(s_{n})} \cdot Q_{\sigma}^{\chi}(s_{n}) \\
&= e^{\epsilon C'} \cdot e^{\delta C} \cdot e^{-\delta ||\gamma||_{\sigma}}.
\end{align*} Since $\epsilon > 0$ was arbitrary and the constant $C'$ depends only on $\gamma$ and not on $\epsilon$, sending $\epsilon \rightarrow 0$ gives
\begin{align*}
\nu( \mathcal{S}_{\frac{t_{0}}{2}}(\gamma)) \leq e^{\delta C} \cdot e^{-\delta ||\gamma||_{\sigma}}.
\end{align*} Notice that the constant $C$ depends only on $\epsilon_{0}$, not on the particular choice of $\gamma \in \mathcal{T}$. Since $\gamma \in \mathcal{T}$ was arbitrary, the lemma holds with $C_{1} := e^{\delta C}$. 
\end{proof}
We need one more observation before we can establish (\ref{asymptotic}). Let $\{\gamma_{n}\}_{n \geq 0} \subset \mathcal{T}$ be any geodesic ray in the Bishop--Jones semigroup $\mathcal{T}$ starting at the identity element $\gamma_{0} = \mathrm{id}$. Let $t_{0} > 0$ be as in the proof of Proposition \ref{BJtree} and recall from (\ref{EstOnTree}) that we have $d(\gamma_{n}^{-1}, \gamma_{n}^{-1} \gamma_{n+1}) > t_{0}$ for all $n \geq 1$, where $\gamma_{n}^{-1} \gamma_{n+1} \in S$ (the set generating $\mathcal{T}$), by construction. Let $C = C(t_{0})$ be the constant furnished by part (3) of Proposition \ref{CocycleProperties} and define $D := \big( \max_{s \in S} ||s||_{\sigma} + C \big) / 2$. Then,
\begin{align*}
    ||\gamma_{n+1}||_{\sigma} \leq ||\gamma_{n}||_{\sigma} + ||\gamma_{n}^{-1} \gamma_{n+1}||_{\sigma} + C \leq ||\gamma_{n}||_{\sigma} + 2D.
\end{align*} Given $R > ||\gamma_{0}||_{\sigma} + D$, define
\begin{align*}
    \mathcal{T}(R) := \{ \gamma \in \mathcal{T} : R-D \leq ||\gamma||_{\sigma} \leq R+D \}.
\end{align*} Let $n_{0}$ be the largest integer so that $||\gamma_{n_{0}}||_{\sigma} < R-D$. Then,
\begin{align*}
R-D \leq ||\gamma_{n_{0}+1}||_{\sigma} \leq ||\gamma_{n_{0}}||_{\sigma} + 2D < R+D,
\end{align*} hence $\gamma_{n_{0}+1} \in \mathcal{T}(R)$. In other words, we have shown that the set $\mathcal{T}(R)$ contains an element of any geodesic ray in $\mathcal{T}$ starting at $\gamma_{0} = \mathrm{id}$. We now prove the aforementioned estimate (\ref{asymptotic}) needed for the proof of Theorem \ref{SemigroupExpGap}.
\begin{prop}
\label{SemigroupExpGrowth}
There exists a constant $C \geq 1$ so that
\begin{align*}
    n(R) \geq \frac{1}{C} e^{\delta_{\sigma}(\mathcal{T})R}
\end{align*} for all $R \geq ||\mathrm{id}||_{\sigma}$.
\end{prop}
\begin{proof}
Fix $R > ||\gamma_{0}||_{\sigma} + D = ||\mathrm{id}||_{\sigma} + D$. By the proof of Lemma \ref{UniformlyConical}, given any $u \in E$, there exists a geodesic ray $\{\gamma_{n}\} \subset \mathcal{T}$ so that $\{u\} = \bigcap_{n=1}^{\infty} \mathcal{S}_{\frac{t_{0}}{2}}(\gamma_{n})$. By the previous discussion, $\gamma_{n} \in \mathcal{T}(R)$ for some $n$. Since $u \in E$ was arbitrary, we have $E \subset \bigcup_{\gamma \in \mathcal{T}(R)} \mathcal{S}_{\frac{t_{0}}{2}}(\gamma)$. 
Let $\nu$ be the Borel probability measure and $C_{1} > 1$ be the constant provided by Lemma \ref{ShadowUpperBound}. Then $\nu$ is supported on $E$ and,
\begin{align*}
    \nu(\mathcal{S}_{\frac{t_{0}}{2}}(\gamma)) \leq C_{1}e^{-\delta_{\sigma}(\mathcal{T})||\gamma||_{\sigma}}
\end{align*} for all $\gamma \in \mathcal{T}$. Hence,
\begin{align*}
C_{1} n(R+D) e^{-\delta_{\sigma}(\mathcal{T})(R-D)} &\geq C_{1} \sum_{\gamma \in \mathcal{T}(R)} e^{-\delta_{\sigma}(\mathcal{T}) (R-D)} \\
&\geq \sum_{\gamma \in \mathcal{T}(R)} \nu \big( \mathcal{S}_{\frac{t_{0}}{2}}(\gamma) \big) \\
&\geq \nu(E) = 1,
\end{align*} and therefore
\begin{align*}
n(R+D) \geq \frac{1}{C_{1}} e^{\delta_{\sigma}(\mathcal{T})(R-D)} = \frac{1}{C_{1} e^{2D\delta_{\sigma}(\mathcal{T})}} e^{\delta_{\sigma}(\mathcal{T})(R+D)} \geq \frac{1}{C} e^{\delta_{\sigma}(\mathcal{T})(R+D)},
\end{align*} where $C = C_{1}e^{2D\delta_{\sigma}(\mathcal{T})}$. By enlarging $C$ if necessary, we obtain 
\begin{align*}
n(R) \geq \frac{1}{C}e^{\delta_{\sigma}(\mathcal{T})R}
\end{align*}
for all $R \geq ||\mathrm{id}||_{\sigma}$, as desired.
\end{proof}
\begin{cor}
\label{PoincareSeriesDiv}
The Poincar\'e series $ Q_{\sigma}^{\mathcal{T}}(s) = \sum_{\gamma \in \mathcal{T}} e^{-s ||\gamma||_{\sigma}}$ diverges at $s = \delta_{\sigma}(\mathcal{T})$.
\end{cor}
\begin{proof}
Notice that the Poincar\'e series 
\begin{align*}
Q_{\sigma}^{\mathcal{T}}(s) = \sum_{\gamma \in \mathcal{T}} e^{-s ||\gamma||_{\sigma}}
\end{align*} converges if and only if the series
\begin{align*}
F(s) = \sum_{k=0}^{\infty} \big(n(k) - n(k-1) \big) e^{-sk}
\end{align*} converges (since the ratio of the two series is bounded between two strictly positive constants). The series $F(s)$ converges if and only if the series 
\begin{align*}
G(s) = \sum_{k=0}^{\infty} n(k) e^{-sk}
\end{align*} converges (because $n(k) = \sum_{i=1}^{k-1} n(i+1) - n(i)$). Thus it suffices to show that the series $G(s)$ diverges at $s = \delta_{\sigma}(\mathcal{T})$. Let $s > \delta_{\sigma}(\mathcal{T})$. By Proposition \ref{SemigroupExpGrowth}, 
\begin{align*}
G(s) &= \sum_{k=0}^{\infty} n(k) e^{-sk}  \\
&\geq \sum_{k = \lfloor ||\mathrm{id}||_{\sigma} \rfloor + 1}^{\infty} n(k) e^{-sk} \\
&\geq \frac{e^{-(s-\delta_{\sigma}(\mathcal{T}))(\lfloor ||\mathrm{id}||_{\sigma} \rfloor + 1)}}{C} \sum_{k=0}^{\infty} e^{-(s-\delta_{\sigma}(\mathcal{T}))k}
\\
&= e^{-(s-\delta_{\sigma}(\mathcal{T}))(\lfloor ||\mathrm{id}||_{\sigma} \rfloor + 1)} \cdot \frac{1}{C(1 - e^{-(s-\delta_{\sigma}(\mathcal{T}))})}.
\end{align*} 
Hence $Q_{\sigma}^{\mathcal{T}}(s)$ diverges at $s = \delta_{\sigma}(\mathcal{T})$, as desired.
\end{proof}
We now turn to showing that $E \cup E'$ is a proper subset of $\Lambda(\Gamma)$. 
\begin{lem}
\label{propersubset}
The set $E \cup E'$ is a proper subset of $\Lambda(\Gamma)$.
\end{lem}
\begin{proof}
We use the same notation as we did throughout the proof of Proposition \ref{BJtree}. Let $y \in \Lambda(\Gamma)$ and $\epsilon_{0} > 0$ be as in point (ii) in the beginning of the proof of Proposition \ref{BJtree}. We will show that $y \notin E \cup E'$. By construction (see once again point (ii) in the proof of the aforementioned proposition), we have $y \in B_{s_{2}}(x) \smallsetminus \overline{B_{s_{1}}(x)}$. Notice that (\ref{TreePf7}) says exactly that the set $\mathcal{T}^{-1}$ lies a distance greater than $2t_{0}$ from the point $x$. Thus $E' \subset M \smallsetminus B_{2t_{0}}(x)$, and since $s_{2} < 2t_{0}$, we have $y \notin E'$. \par 
It remains to show that $y \notin E$. By construction,
\begin{align*}
    t_{0} + \epsilon_{0} < \frac{2}{3} d(y,x) + \frac{1}{12} d(y,x) = s_{1}. 
\end{align*} We claim that
\begin{align*}
E \subset \overline{B_{t_{0} + \epsilon_{0}}(x)} \subset B_{s_{1}}(x). 
\end{align*} 
Indeed, given $\alpha \in S^{m+1}$, $m \geq 1$, write $\alpha = \gamma \beta$ where $\gamma \in S$ and $\beta \in S^{m}$. By Lemma \ref{obs1} and (\ref{inclu1}), we have
\begin{align*}
\alpha \in \gamma \big( (\Gamma \sqcup M) \smallsetminus B_{\epsilon_{0}}(\gamma^{-1}) \big) \subset B_{\epsilon_{0}} (\gamma) \subset B_{t_{0} + \epsilon_{0}}(x).
\end{align*} It follows that $E \subset \overline{B_{t_{0} + \epsilon_{0}}(x)} \subset B_{s_{1}}(x)$. But as $y \notin \overline{B_{s_{1}}(x)}$, we see that $y \notin E$. Thus $y \in \Lambda(\Gamma) \smallsetminus (E \cup E')$, concluding the proof.  
\end{proof}
We now give the proof of Theorem \ref{SemigroupExpGap}.
\begin{proof} [Proof of Theorem \ref{SemigroupExpGap}]
By the preceding lemma, $E \cup E'$ is a proper closed subset of $\Lambda(\Gamma)$. Fix an open set $U \subset M$ so that $U \cap \Lambda(\Gamma) \neq \emptyset$ and $\overline{U} \cap (E \cup E') = \emptyset$. We claim that
\begin{align*}
N := \# \{\gamma \in \mathcal{T} : \gamma U \cap U \neq \emptyset \} < \infty.
\end{align*} If not, then there exist a sequence of distinct elements $\{\gamma_{n}\} \subset \mathcal{T}$ and a sequence $\{x_{n}\} \subset U$ so that $\{\gamma_{n}x_{n}\} \subset U$. Passing to subsequences, we may assume without loss of generality that $\gamma_{n} \rightarrow a \in E$ and $\gamma_{n}^{-1} \rightarrow b \in E'$. Then $\gamma_{n} \big|_{M \smallsetminus \{b\}} \rightarrow a$ locally uniformly. Since $b \notin \overline{U}$ and $\overline{U} \subset M$ is compact, we conclude that $\gamma_{n}(\overline{U}) \rightarrow a$. But then $\gamma_{n} (\overline{U}) \cap \overline{U} = \emptyset$ for all $n$ sufficiently large, which gives the desired contradiction. 
\par 
By Theorem $\ref{PSmeasure}$, there exists a constant $C_{1} := 2 \kappa \delta_{\sigma}(\Gamma) > 0$ and a $C_{1}$-coarse $\sigma$-Patterson--Sullivan measure $\mu$ of dimension $\delta_{\sigma}(\Gamma)$ for $\Gamma$ which is supported on the limit set $\Lambda(\Gamma)$. We therefore have 
\begin{align*}
    e^{-C_{1}-\delta_{\sigma}(\Gamma) \sigma(\gamma^{-1}, \cdot)} \leq \frac{d(\gamma_{*}\mu)}{d \mu} \leq e^{C_{1}-\delta_{\sigma}(\Gamma) \sigma(\gamma^{-1}, \cdot)}
\end{align*} $\mu$-almost everywhere. Suppose for a contradiction that $\delta_{\sigma}(\mathcal{T}) = \delta_{\sigma}(\Gamma)$. Since $\Gamma$ acts minimally on $\Lambda(\Gamma)$, we have $\mu(U) > 0$. As the Hausdorff distance between $\overline{U}$ and $\mathcal{T}$ in $\Gamma \sqcup M$ is positive, the expanding property of the coarse-cocycle $\sigma$ guarantees the existence of a constant $C_{2} > 0$ so that $\big| \sigma(\gamma,x) - ||\gamma||_{\sigma} \big| \leq C_{2}$ for all $x \in \overline{U}$ and $\gamma \in \mathcal{T}$. It follows that
\begin{align*}
N+1 &\geq \sum_{\gamma \in \mathcal{T}} \mu(\gamma U) = \sum_{\gamma \in \mathcal{T}} (\gamma^{-1})_{*} \mu(U) \geq e^{-C_{1}} \sum_{\gamma \in \mathcal{T}} \int_{U} e^{-\delta_{\sigma}(\Gamma) \sigma(\gamma, x)} \ d\mu(x) \\
&\geq e^{-C_{1}} \sum_{\gamma \in \mathcal{T}} \int_{U} e^{-\delta_{\sigma}(\Gamma) (||\gamma||_{\sigma} + C_{2})} \ d\mu(x) \\
&= \frac{\mu(U)}{e^{C_{1} + \delta_{\sigma}(\mathcal{T})C_{2}}}  \sum_{\gamma \in \mathcal{T}} e^{-\delta_{\sigma}(\mathcal{T}) ||\gamma||_{\sigma}} = \infty,
\end{align*} where the last equality follows from Corollary \ref{PoincareSeriesDiv}. This gives the desired contradiction, and therefore we must have $\delta_{\sigma}(\mathcal{T}) < \delta_{\sigma}(\Gamma)$.
\end{proof}
\section{Background on Lie Theory}
\label{BackgroundSection}
In this section, we record the necessary background on Lie theory, transverse groups, and Anosov semigroups that we will need in order to apply the results of the previous section to prove Theorems \ref{TransverseCritExp} and \ref{NoGapRank1'} in the coming section.
\subsection{Basics of Lie Theory}
Let $G$ be a connected semisimple real Lie group without compact factors and with finite center and let $\mathfrak{g}$ denote its Lie algebra. Let $b$ denote the Killing form of $\mathfrak{g}$ and fix a Cartan involution $\tau$ of $\mathfrak{g}$; that is, an involution of $\mathfrak{g}$ for which the bilinear pairing $\langle \cdot, \cdot \rangle$ on $\mathfrak{g}$ defined by $\langle X, Y \rangle := -b(X, \tau (Y))$ is an inner product. Then $\mathfrak{g}$ decomposes as $\mathfrak{g} = \mathfrak{k} \oplus \mathfrak{p}$, where $\mathfrak{k}$ and $\mathfrak{p}$ are the 1 and -1 eigenspaces of $\tau$. The subalgebra $\mathfrak{k}$ is a maximal compact Lie subalgbra of $\mathfrak{g}$ and we denote by $K \subset G$ the maximal compact Lie subgroup of $G$ whose Lie algebra is $\mathfrak{k}$.

Fix a maximal abelian subspace $\mathfrak{a} \subset \mathfrak{p}$, known as a $\textit{Cartan subspace}$, which is unique up to conjugation. The Lie algebra $\mathfrak{g}$ then decomposes as 
\begin{align*}
    \mathfrak{g} = \mathfrak{g}_{0} \oplus \bigoplus_{\alpha \in \Sigma} \mathfrak{g}_{\alpha},
\end{align*} which is called the $\textit{restricted root space decomposition}$ associated to $\mathfrak{a}$; in this decomposition, for $\alpha \in \mathfrak{a}^{*}$, we define
\begin{align*}
    \mathfrak{g}_{\alpha} := \{X \in \mathfrak{g} : [H,X] = \alpha (H)X \ \ \mathrm{for \ all} \ H \in \mathfrak{a} \},
\end{align*} and call
\begin{align*}
    \Sigma := \{ \alpha \in \mathfrak{a}^{*} \smallsetminus \{0\} : \mathfrak{g}_{\alpha} \neq 0 \}
\end{align*} the set of $\textit{restricted roots}$. Now fix an element $H_{0} \in \mathfrak{a}$ so that $\alpha (H_{0}) \neq 0$ for all $\alpha \in \Sigma$, and let
\begin{align*}
\Sigma^{+} := \{ \alpha \in \Sigma : \alpha (H_{0}) > 0 \} \ \ \mathrm{and} \ \ \Sigma^{-} := - \Sigma^{+}.
\end{align*} Notice that $\Sigma = \Sigma^{+} \sqcup \Sigma^{-}$. We write $\Delta \subset \Sigma^{+}$ for the set of $\textit{simple restricted roots}$, which, by definition, consists of all the elements of $\Sigma^{+}$ which cannot be written as a non-trivial positive integer linear combination of elements in $\Sigma^{+}$. As $\Sigma$ is an abstract root system on $\mathfrak{a}^{*}$, it follows that $\Delta$ is a basis of $\mathfrak{a}^{*}$ and every $\alpha \in \Sigma^{+}$ is a non-negative integral linear combination of elements in $\Delta$. See for instance Chapter II of Knapp's book \cite{K} for more details.
\subsubsection{The Weyl Group, Opposition Involution, and Cartan Projection}
 The \emph{Weyl group} of $\mathfrak{a}$ is given by $\mathcal{W} := N_{K}(\mathfrak{a}) / Z_{K}(\mathfrak{a})$, where $N_{K}(\mathfrak{a}) \subset K$ is the normalizer of $\mathfrak{a}$ in $K$ and $Z_{K}(\mathfrak{a}) \subset K$ is the centralizer of $\mathfrak{a}$ in $K$. The Weyl group is a finite group generated by reflections of $\mathfrak{a}$ (with respect to the inner product $\langle \cdot, \cdot \rangle$) about the kernels of the simple restricted roots in $\Delta$. Hence, $\mathcal{W}$ acts transitively on the set of $\textit{Weyl chambers}$, which are the closures of the connected components of 
\begin{align*}
    \mathfrak{a} - \bigcup_{\alpha \in \Sigma} \ker \alpha.
\end{align*} We call the Weyl chamber 
\begin{align*}
    \mathfrak{a}^{+} := \{X \in \mathfrak{a} : \alpha (X) \geq 0 \ \ \mathrm{for \ all} \ \alpha \in \Delta \},
\end{align*} the $\textit{positive Weyl chamber}$. In the Weyl group $\mathcal{W}$, there is a unique element $w_{0}$, called the $\textit{longest element}$, with the property that $w_{0}(\mathfrak{a}^{+}) = - \mathfrak{a}^{+}$. Thus the longest element allows us to define an involution $\iota : \mathfrak{a} \rightarrow \mathfrak{a}$, $H \mapsto - w_{0} \cdot H$, which is called the $\textit{opposition involution}$. It induces an involution of $\Sigma$ preserving $\Delta$, denoted by $\iota^{*}$, defined by $\iota^{*}(\alpha) = \alpha \circ \iota$ for all $\alpha \in \Delta$. Moreover, if $k_{0} \in N_{K}(\mathfrak{a})$ denotes a representative of the longest element $w_{0} \in \mathcal{W}$, then
\begin{align}
\label{OppInvAct}
\mathrm{Ad}(k_{0}) \mathfrak{g}_{\alpha} = \mathfrak{g}_{-\iota^{*}(\alpha)}
\end{align} for all $\iota \in \Sigma$.
\par 
Let $\kappa : G \rightarrow \mathfrak{a}^{+}$ denote the $\textit{Cartan projection}$, that is $\kappa (g) \in \mathfrak{a}^{+}$ is the unique element so that 
\begin{align*}
    g = k \exp (\kappa (g)) k'
\end{align*} for some $k,k' \in K$. We note that $k, k' \in K$ need not be unique. Such a decomposition of $g \in G$ is called a $KA^{+}K$ decomposition (see Theorem $7.39$ of \cite{K}). Notice that since $\iota(-\mathfrak{a}^{+}) = \mathfrak{a}^{+}$, we have $\iota(\kappa(g)) = \kappa(g^{-1})$ for all $g \in G$. Recall also that when $G = \mathrm{SL}(d, \mathbb{R})$, the Cartan projection is given by 
\begin{align*}
    \kappa(g) = \mathrm{diag} (\log \sigma_{1}(g), \dots, \log \sigma_{d}(g)),
\end{align*} where $\sigma_{1}(g) \geq \cdots \geq \sigma_{d}(g) > 0$ are the singular values of $g$.

\subsubsection{Parabolic Subgroups and Flag Manifolds}
Given a subset $\theta \subset \Delta$, consider the subset of positive restricted roots $\Sigma_{\theta}^{+} := \Sigma^{+} \smallsetminus (\mathrm{Span} (\Delta \smallsetminus \theta))$. We then define the $\textit{standard parabolic subgroup associated to} \ \theta$, denoted by $P_{\theta} = P_{\theta}^{+} \subset G$, to be the normalizer in $G$ of
\begin{align*}
\mathfrak{u}_{\theta} = \mathfrak{u}_{\theta}^{+} := \bigoplus_{\alpha \in \Sigma_{\theta}^{+}} \mathfrak{g}_{\alpha}.
\end{align*} The $\textit{standard partial flag variety associated to} \ \theta$ is 
\begin{align*}
\mathcal{F}_{\theta} = \mathcal{F}_{\theta}^{+} := G / P_{\theta}.
\end{align*} We similarly define the $\textit{standard parabolic subgroup opposite to} \ P_{\theta}$, denoted by $P_{\theta}^{-}$, to be the normalizer in $G$ of 
\begin{align*}
\mathfrak{u}_{\theta}^{-} := \bigoplus_{\alpha \in \Sigma_{\theta}^{+}} \mathfrak{g}_{- \iota^{*}(\alpha)},
\end{align*} and the $\textit{standard partial flag variety opposite to} \ \mathcal{F}_{\theta}$ by
\begin{align*}
\mathcal{F}_{\theta}^{-} := G / P_{\theta}^{-}.
\end{align*} Using (\ref{OppInvAct}), we see that if $k_{0} \in N_{K}(\mathfrak{a})$ is a representative of the longest element $w_{0} \in \mathcal{W}$, then
\begin{align*}
k_{0}P_{\theta}^{\pm}k_{0}^{-1} = k_{0}^{-1}P_{\theta}^{\pm}k_{0} = P_{\iota^{*}(\theta)}^{\mp}
\end{align*}
We say that two flags $F_{1} \in \mathcal{F}_{\theta}^{+}$ and $F_{2} \in \mathcal{F}_{\theta}^{-}$ are $\textit{transverse}$ if the pair $(F_{1}, F_{2})$ is contained in the $G$-orbit of $(P_{\theta}^{+}, P_{\theta}^{-})$ in $\mathcal{F}_{\theta}^{+} \times \mathcal{F}_{\theta}^{-}$. Given any flag $F \in \mathcal{F}_{\theta}^{\pm}$, we denote by $\mathcal{Z}_{F} \subset \mathcal{F}_{\theta}^{\mp}$ the set of flags that are not transverse to $F$. As the set of transverse pairs in $\mathcal{F}_{\theta}^{+} \times \mathcal{F}_{\theta}^{-}$ is an open and dense subset, $\mathcal{Z}_{F}$ is a closed subset with empty interior. Furthermore, $\mathcal{Z}_{F} = \mathcal{Z}_{F'}$ if and only if $F = F'$.
\par 
If $\theta \subset \Delta$ is $\textit{symmetric}$, that is $\iota^{*}(\theta) = \theta$, then the fact that $k_{0} P_{\theta} k_{0}^{-1} = P_{\theta}^{-}$ allows us to identify $\mathcal{F}_{\theta}$ and $\mathcal{F}_{\theta}^{-}$ via the map
\begin{align*}
gP_{\theta}^{-} \mapsto gk_{0}^{-1}P_{\theta}.
\end{align*} Thus it makes sense to say whether or not two elements in $\mathcal{F}_{\theta}$ are transverse. Precisely, two flags $g_{1}P_{\theta}$ and $g_{2}P_{\theta}$ in $\mathcal{F}_{\theta}$ are $\textit{transverse}$ if and only if there exists some $g \in G$ so that $gg_{1} \in P_{\theta}$ and $gg_{2}k_{0} \in P_{\theta}^{-}$. Abusing notation, given a flag $F \in \mathcal{F}_{\theta}$, we write $\mathcal{Z}_{F} \subset \mathcal{F}_{\theta}$ for the set of flags that are not transverse to $F$. 
\par 
When $\theta = \Delta$ is the full set of simple restricted roots, its associated parabolic subgroup $P_{\Delta}$, which we denote simply by $P$, is a minimal (with respect to inclusion) parabolic subgroup called the $\textit{Borel subgroup}$. The associated flag variety $\mathcal{F}_{\Delta}$, henceforth written simply $\mathcal{F} = G/P$, is the $\textit{Furstenberg boundary}$.
\subsubsection{Partial Cartan Subspaces and the Partial Iwasawa Cocycle}
Given a non-empty subset $\theta \subset \Delta$, we define
\begin{align*}
    \mathfrak{a}_{\theta} &= \bigcap_{\alpha \in \Delta \smallsetminus \theta} \ker \alpha, \ \ \ \ \mathfrak{a}_{\theta}^{+} = \mathfrak{a}_{\theta} \cap \mathfrak{a}^{+}, \\
    A_{\theta} &= \exp(\mathfrak{a}_{\theta}), \ \mathrm{and} \ \ A_{\theta}^{+} = \exp(\mathfrak{a}_{\theta}^{+}).
\end{align*} Let
\begin{align*}
    p_{\theta} : \mathfrak{a} \rightarrow \mathfrak{a}_{\theta}
\end{align*} denote the projection of $\mathfrak{a}$ onto $\mathfrak{a}_{\theta}$ invariant under all $w \in \mathcal{W}$ fixing $\mathfrak{a}_{\theta}$ pointwise. The dual space $\mathfrak{a}_{\theta}^{*} = \mathrm{Hom}(\mathfrak{a}_{\theta}, \mathbb{R})$ can in fact be identified with the subspace of $p_{\theta}$-invariant linear functionals on $\mathfrak{a}$:
\begin{align*}
\mathfrak{a}_{\theta}^{*} = \{ \phi \in \mathfrak{a}^{*} : \phi \circ p_{\theta} = \phi \}.
\end{align*} Thus the $\textit{partial Cartan projection associated to} \ \theta$
\begin{align*}
\kappa_{\theta} = p_{\theta} \circ \kappa : G \rightarrow \mathfrak{a}_{\theta}
\end{align*} has the property that, for all $\phi \in \mathfrak{a}_{\theta}^{*}$ and $g \in G$,
\begin{align}
\label{ComposeCartan}
\phi(\kappa_{\theta}(g)) = \phi(\kappa(g)).
\end{align} Let $N := \mathrm{exp}(\mathfrak{u}_{\Delta})$. The $\textit{Iwasawa decomposition}$ is the statement that the map
\begin{align*}
K \times A \times N &\rightarrow G \\
(k,a,n) &\mapsto kan,
\end{align*} is a diffeomorphism; see Proposition 6.46 of \cite{K}. In \cite{Q1}, Quint used the Iwasawa decomposition to define the $\textit{Iwasawa cocycle}$
\begin{align*}
B : G \times \mathcal{F} \rightarrow \mathfrak{a},
\end{align*} which has the defining property that $gk \in K \cdot \mathrm{exp}(B(g,F)) \cdot N$ for all $(g,F) \in G \times \mathcal{F}$, where $k \in K$ is such that $F = kP$. Intuitively, the Iwasawa cocycle can be thought of as a vector valued analog of the Busemann cocycle in the rank 1 setting, and was used by Quint to define higher-rank Patterson--Sullivan measures. \par Given $\theta \subset \Delta$, we have $P \subset P_{\theta}$, hence the identity map on $G$ induces a surjection $\Pi_{\theta} : \mathcal{F} \rightarrow \mathcal{F}_{\theta}$. We can then define the $\textit{partial Iwasawa cocycle}$ to be the map 
\begin{align*}
B_{\theta} : G \times \mathcal{F}_{\theta} \rightarrow \mathfrak{a}_{\theta}
\end{align*} given by $B_{\theta}(g,F) = p_{\theta}(B(g,F'))$ for some (all) $F' \in \Pi_{\theta}^{-1}(F)$. Quint showed that this is a well-defined cocycle, meaning that 
\begin{align*}
B_{\theta}(gh,F) = B_{\theta}(g,hF) + B_{\theta}(h,F)
\end{align*} for all $g,h \in G$ and $F \in \mathcal{F}_{\theta}$. See section $6$ of \cite{Q1} for more details. 
\par 
Suppose now that $\theta \subset \Delta$ is symmetric, that is, $\iota^{*}(\theta) = \theta$. Let $\mathcal{F}_{\theta}^{(2)} \subset \mathcal{F}_{\theta} \times \mathcal{F}_{\theta}$ denote the subset of transverse flags. There is a smooth map $G_{\theta} : \mathcal{F}_{\theta}^{(2)} \rightarrow \mathfrak{a}_{\theta}$ that satisfies
\begin{align*}
G_{\theta}(gF, gF') - G_{\theta}(F, F') = \iota \circ B_{\theta}(g,F) + B_{\theta}(g,F'),
\end{align*} for all $g \in G$ and $(F,F') \in \mathcal{F}_{\theta}^{(2)}$. See for instance the end of section 2 of \cite{KOW2}.
\subsubsection{Critical Exponents}
Let $\Gamma \subset G$ be any discrete sub(semi)group and let $\theta \subset \Delta$. In higher rank, the (partial) Cartan projection should intuitively be thought of as a ``vector valued distance,'' and as such allows us to measure asymptotic growth rates of certain quantities related to $\Gamma$. More precisely, given $\phi \in \mathfrak{a}_{\theta}^{*}$, let $Q_{\Gamma}^{\phi}(s)$ denote the Poincar\'e series
\begin{align*}
    Q_{\Gamma}^{\phi}(s) = \sum_{\gamma \in \Gamma} e^{-s \phi(\kappa_{\theta}(\gamma))}.
\end{align*} As before, the $\textit{critical exponent}$ $\delta^{\phi}(\Gamma)$ is the abscissa of convergence of this Poincar\'e series, i.e.,
\begin{align*}
    \delta^{\phi}(\Gamma) = \inf \{s > 0 : Q_{\Gamma}^{\phi}(s) < \infty \} \in [0, \infty].
\end{align*} Equivalently, 
\begin{align*}
\delta^{\phi}(\Gamma) = \limsup_{T 
\rightarrow \infty} \frac{1}{T} \log \# \{\gamma \in \Gamma : \phi(\kappa_{\theta}(\gamma)) < T \}.
\end{align*}
\subsection{Transverse Groups}
In this section, we recall the definition of a transverse subgroup of a connected semisimple real Lie group $G$ with finite center and no compact factors. Transverse groups are a family of discrete subgroups of semisimple Lie groups, which contain both the class of Anosov and relatively Anosov groups, as well as all discrete subgroups of rank one Lie groups. 
\par 
For the remainder of this section, we assume that $\theta \subset \Delta$ is symmetric. Following the notation of \cite{GGKW}, we define a map
\begin{align*}
U_{\theta} : G \rightarrow \mathcal{F}_{\theta}
\end{align*} by fixing a $KA^{+}K$-decomposition $g = k_{g} \mathrm{exp}(\kappa(g)) k_{g}'$ for each $g \in G$ and setting $U_{\theta}(g) = k_{g}P_{\theta}$. If $\alpha (\kappa (g)) > 0$ for all $\alpha \in \theta$, then $U_{\theta}(g)$ is independent of the choice of $KA^{+}K$-decomposition (see for instance Theorem 1.1 in Chapter IX of \cite{H}), and so $U_{\theta}$ is continuous on the set 
\begin{align*}
\{g \in G : \alpha (\kappa (g)) > 0 \ \mathrm{for \ all} \ \alpha \in \theta \}.
\end{align*} 
\par 
We say that a subgroup $\Gamma$ of $G$ is $P_{\theta}$-$\textit{divergent}$ if
\begin{align*}
\lim_{n \to \infty} \min_{\alpha \in \theta} \alpha(\kappa(\gamma_{n})) = \infty
\end{align*} for any sequence $\{\gamma_{n}\}$ of distinct elements of $\Gamma$. A $P_{\theta}$-divergent group is discrete and its $\theta$-$\textit{limit set}$ is the set of accumulation points in $\mathcal{F}_{\theta}$ of $\{U_{\theta}(\gamma) : \gamma \in \Gamma\}$. Using Proposition 2.6 of \cite{CZZ2}, one can show that $\Lambda_{\theta}(\Gamma)$ is a closed, $\Gamma$-invariant subset of $\mathcal{F}_{\theta}$. A $P_{\theta}$-divergent subgroup $\Gamma \subset G$ is called $P_{\theta}$-$\textit{transverse}$ if in addition any two distinct flags $F,G \in \Lambda_{\theta}(\Gamma)$ are transverse. We remark that in the literature, divergent and transverse groups are sometimes called $\textit{regular}$ and $\textit{regular antipodal}$, respectively (for instance, in the work of Kapovich--Leeb--Porti \cite{KLP}). An essential feature of transverse groups is that they act on their limit sets as convergence groups.
\begin{prop} [Section 5.1 of \cite{KLP} or Proposition 3.3 of \cite{CZZ1}]
If $\Gamma$ is $P_{\theta}$-transverse, then $\Gamma$ acts on $\Lambda_{\theta}(\Gamma)$ as a convergence group. Moreover, if $\Gamma$ is non-elementary, then $\Gamma$ acts on $\Lambda_{\theta}(\Gamma)$ minimally, and $\Lambda_{\theta}(\Gamma)$ is perfect.
\end{prop}
The following proposition will allow us to apply Theorem \ref{MainThm} to the setting of transverse groups. 
\begin{prop} [Part (1) of Proposition 10.3 of \cite{BCZZ2}]
\label{ConvtoTransv}
Suppose $\Gamma$ is non-elementary $P_{\theta}$-transverse, $\phi \in \mathfrak{a}_{\theta}^{*}$, and 
\begin{align*}
    \lim_{n \to \infty} \phi (\kappa (\gamma_{n})) = \infty
\end{align*} for any sequence $\{\gamma_{n}\}$ of pairwise distinct elements of $\Gamma$. Define cocycles $\sigma_{\phi}, \bar{\sigma}_{\phi} : \Gamma \times \Lambda_{\theta}(\Gamma) \rightarrow \mathbb{R}$ by 
\begin{align*}
\sigma_{\phi}(\gamma, F) = \phi (B_{\theta}(\gamma, F)) \ \ \mathrm{and} \ \ \bar{\sigma}_{\phi} (\gamma, F) = \iota^{*}(\phi)(B_{\theta}(\gamma, F)).
\end{align*} Then $(\sigma_{\phi}, \bar{\sigma}_{\phi}, \phi \circ G_{\theta})$ is a continuous GPS system for the action of $\Gamma$ on $\Lambda_{\theta}(\Gamma)$. Moreover, one can take the magnitude functions to be
\begin{align*}
    ||\gamma||_{\sigma_{\phi}} = \phi (\kappa (\gamma)) \ \ \mathrm{and} \ \ ||\gamma||_{\bar{\sigma}_{\phi}} = \iota^{*}(\phi)(\kappa(\gamma)).
\end{align*} In particular, $\delta^{\phi}(\Gamma) = \delta_{\sigma_{\phi}}(\Gamma)$.
\end{prop}
\subsection{Anosov Semigroups}
In this section, we discuss the notion of Anosov semigroups recently introduced by Kassel and Potrie in \cite{KP}. The motivation of their paper was to show how certain techniques in the study of discrete subgroups of Lie groups could be fruitful in understanding linear cocycles in dynamics, and vice-versa. One of the main topics of their paper was semigroup representations with a uniform gap in the exponential growth rate of eigenvalues and singular values. In particular, the authors wished to extend the notion of Anosov representations of discrete subgroups of Lie groups -- initially introduced by Labourie in \cite{La} and further developed by Guichard--Wienhard in \cite{GW} -- to semigroups. However, it is not clear how to adapt the original definition of Anosov representations to this more general setting. Instead, motivated by the notion of dominated splittings for linear cocycles, Kassel--Potrie came up with the definition detailed below.
\par 
Let $\Lambda$ be a semigroup with a finite generating subset $S$. That is, any element of $\Lambda$ can be written as a product of elements of elements of $S$. A semigroup may or may not admit an identity element $\mathrm{id}$. For $\gamma \in \Lambda \smallsetminus \{\mathrm{id}\}$, define the $\textit{word length of} \ \gamma \ \textit{with respect to} \ S$ to be
\begin{align*}
|\gamma|_{S} := \min \{k \geq 1 : \gamma = f_{1} \cdots f_{k}, \ \mathrm{where} \ f_{i} \in S \ \mathrm{for \ all} \ 1 \leq i \leq k \},
\end{align*} and set $|\mathrm{id}|_{S} := 0$. Note that if $S'$ is another finite generating set of $\Lambda$, then there exists a constant $M \geq 1$ so that
\begin{align}
\label{genset}
M^{-1} |\gamma|_{S'} \leq |\gamma|_{S} \leq M |\gamma|_{S'}
\end{align}
for all $\gamma \in \Lambda$.
\begin{defn}
Let $G$ be a connected semisimple real Lie group and $\theta \subset \Delta$ a collection of simple restricted roots. A semigroup homomorphism $\rho : \Lambda \rightarrow G$ is said to be $P_{\theta}$-$\textit{Anosov}$ if there exists constants $C,c > 0$ so that 
\begin{align}
\label{AnosovSemi}
\alpha (\kappa (\rho (\gamma))) \geq C |\gamma|_{S} - c 
\end{align} for all $\gamma \in \Lambda$ and $\alpha \in \theta$.
\end{defn} If $\Lambda \subset G$ is a semigroup and satisfies (\ref{AnosovSemi}) with $\rho$ being the inclusion $\Lambda \hookrightarrow G$, then we say that $\Lambda$ is a $P_{\theta}$-$\textit{Anosov subsemigroup of} \ G$. Note that by (\ref{genset}), this definition is independent of the choice of finite generating set for $\Lambda$.
\section{Critical Exponents of Transverse Groups and Relations to Corlette's Gap Theorem}
\label{Applications}
In this section, we provide two applications of Theorem \ref{MainThm}. The first proves Theorem \ref{TransverseCritExpIntro} from the introduction.
\begin{thm}
\label{TransverseCritExp}
Suppose $\Gamma$ is a non-elementary $P_{\theta}$-transverse group and $\phi \in \mathfrak{a}_{\theta}^{*}$ is such that $\phi(\kappa(\gamma_{n})) \rightarrow \infty$ for any sequence $\{\gamma_{n}\}$ of pairwise distinct elements of $\Gamma$. Then there exists a sequence $\{\Gamma_{n}\}_{n \geq 1}$ of free $P_{\theta}$-Anosov subsemigroups of $\Gamma$ so that 
\begin{itemize}
    \item[(1)] $\delta^{\phi}(\Gamma_{n}) < \delta^{\phi}(\Gamma)$ for all $n \geq 1$, and
    \item[(2)] $\lim_{n \to \infty} \delta^{\phi}(\Gamma_{n}) = \delta^{\phi}(\Gamma)$.
\end{itemize} 
\end{thm}
\begin{proof}
By Proposition $\ref{ConvtoTransv}$, $(\sigma_{\phi}, \bar{\sigma}_{\phi}, \phi \circ G_{\theta})$ is a continuous GPS system for the action of $\Gamma$ on its limit set $\Lambda_{\theta}(\Gamma)$. In particular, $\sigma_{\phi} : \Gamma \times \Lambda_{\theta}(\Gamma) \rightarrow \mathbb{R}$ is an expanding coarse-cocycle and we may consider the corresponding $\sigma_{\phi}$-magnitude defined by $||\gamma||_{\sigma_{\phi}} = \phi (\kappa (\gamma))$ for all $\gamma \in \Gamma$. Moreover, for any subset $H \subset \Gamma$, we have $\delta^{\phi}(H) = \delta_{\sigma_{\phi}}(H).$ Let $\{\delta_{n}\}_{n \geq 1}$ be a sequence of real numbers so that $\delta_{n} < \delta^{\phi}(\Gamma)$ for all $n$ and $\lim_{n \to \infty} \delta_{n} = \delta^{\phi}(\Gamma)$. For each $n \geq 1$, let $\Gamma_{n}$ be the free Bishop--Jones subsemigroup of $\Gamma$ associated to $\delta_{n}$ provided by Theorem $\ref{MainThm}$. By properties (1) and (4) of Theorem \ref{MainThm}, we have 
\begin{align*}
 \delta_{n} \leq   \delta^{\phi}(\Gamma_{n}) < \delta^{\phi}(\Gamma),
\end{align*}
for all $n \geq 1$, and therefore
\begin{align*}
    \lim_{n \to \infty} \delta^{\phi}(\Gamma_{n}) = \delta^{\phi}(\Gamma).
\end{align*} This proves items (1) and (2) in the statement of the theorem. It only remains to prove that each $\Gamma_{n}$ is $P_{\theta}$-Anosov. Fix an arbitrary $n \geq 1$ and let $S_{n}$ be the corresponding finite generating subset of $\Gamma_{n}$. Since $\Gamma$ is $P_{\theta}$-transverse, it is in particular $P_{\theta}$-divergent, so $\lim_{m \to \infty} \alpha (\kappa(\gamma_{m})) = \infty$ for any $\alpha \in \theta$ and any sequence $\{\gamma_{m}\}$ of distinct elements in $\Gamma$. Thus Proposition $\ref{ConvtoTransv}$ tells us once again that $\sigma_{\alpha} : \Gamma \times \Lambda_{\theta}(\Gamma) \rightarrow \mathbb{R}$ is an expanding coarse-cocycle for each $\alpha \in \theta$, and that we can choose the corresponding magnitude functions to be $||\gamma||_{\sigma_{\alpha}} = \alpha(\kappa(\gamma))$. By property (2) of Theorem \ref{MainThm}, for each $\alpha \in \theta$, there exist constants $A_{\alpha}, a_{\alpha} > 0$ so that 
\begin{align*}
\alpha(\kappa(\gamma)) = ||\gamma||_{\sigma_{\alpha}} \geq A_{\alpha} |\gamma|_{S_{n}} - a_{\alpha}
\end{align*} for all $\gamma \in \Gamma_{n}$. Setting $C := \min \{A_{\alpha} : \alpha \in \theta \} > 0$ and $c := \max \{a_{\alpha} : \alpha \in \theta \} > 0$, we obtain 
\begin{align*}
    \alpha(\kappa(\gamma)) \geq C |\gamma|_{S_{n}} - c
\end{align*} for all $\gamma \in \Gamma_{n}$ and $\alpha \in \theta$, which says exactly that $\Gamma_{n}$ is $P_{\theta}$-Anosov. 
\end{proof}
We now prove Theorem \ref{NoGapRank1}, which shows that there is no critical exponent gap for subsemigroups of lattices in $\mathrm{Sp}(n,1)$ and $\mathrm{F}_{4}^{-20}$. Since the class of transverse groups contains all discrete subgroups of rank one Lie groups, this result follows immediately from the previous result. However, there is a proof which does not require the theory of transverse groups, which is the one we present below. \\
\begin{thm} ~\
\label{NoGapRank1'}
\begin{itemize}
    \item[(1)] Let $\Gamma \subset \mathrm{Sp}(n,1)$, $n \geq 2$, be a lattice. Then there exists a sequence $\{\Gamma_{m}\}_{m \geq 1}$ of finitely generated free subsemigroups of $\Gamma$ so that $\delta(\Gamma_{m}) < \delta(\Gamma) = 4n+2$ for all $m \geq 1$ and $\lim_{m \to \infty} \delta(\Gamma_{m}) = 4n+2$.
    \item[(2)] Let $\Gamma \subset \mathrm{F}_{4}^{-20}$ be a lattice. Then there exists a sequence $\{\Gamma_{m}\}_{m \geq 1}$ of finitely generated free subsemigroups of $\Gamma$ so that $\delta(\Gamma_{m}) < \delta(\Gamma) = 22$ for all $m \geq 1$ and $\lim_{m \to \infty} \delta(\Gamma_{m}) = 22$.
\end{itemize}
\end{thm}
\begin{proof}
We prove only (1) as the proof of (2) is entirely analogous. Fix $n \geq 2$, let $\Gamma \subset \mathrm{Sp}(n,1) = \mathrm{Isom}^{+}(\mathbf{H}_{\mathbb{H}}^{n})$ be a lattice, and fix a basepoint $o \in \mathbf{H}_{\mathbb{H}}^{n}$. The quaternionic hyperbolic space $\mathbf{H}_{\mathbb{H}}^{n}$ is a proper geodesic Gromov hyperbolic space, and therefore, as in Example \ref{GromovHypEx}$, (\beta, \beta, G)$ is a GPS system. Taking the $\beta$-magnitude function to be $||\gamma||_{\beta} = d(o, \gamma o)$, we have $\delta_{\beta}(\Gamma) = \delta(\Gamma) = 4n+2$. The result then follows from parts (1) and (4) of Theorem \ref{MainThm}.
\end{proof}
\section{Quasi-Isometric Embeddings of Bishop--Jones Semigroups Into the Symmetric Space}
\label{CoarseGeom}
Fix a non-empty, symmetric subset $\theta \subset \Delta$ and let $\Gamma \subset G$ be $P_{\theta}$-transverse. Write $X = G/K$ for the symmetric space associated to $G$ and let $o = [K] \in G/K$. Fix a $K$-invariant norm $||\cdot||$ on $\mathfrak{g}$ and a metric $d_{X}$ on $X$ induced from the Killing form on $\mathfrak{g}$, so that 
\begin{align*}
d_{X}(go, ho) = ||\kappa(g^{-1}h)||
\end{align*} for all $g,h \in G$. Given a linear functional $\phi \in \mathfrak{a}_{\theta}^{*}$, define a function $\textbf{d}_{\phi} : X \times X \rightarrow \mathbb{R}$ by 
\begin{align*}
    \textbf{d}_{\phi}(go,ho) := \phi (\kappa(g^{-1}h)) = \phi(\kappa_{\theta}(g^{-1}h)).
\end{align*} Since the Cartan projection $\kappa : \mathfrak{a}^{+} \rightarrow \mathbb{R}$ is bi-$K$-invariant, this is a well-defined left $G$-invariant function. We remark that sometimes the function $\mathbf{d}_{\phi}$ is a distance function on the whole symmetric space, and sometimes it is not; see for example \cite{KL1} and \cite{KL2}.
\par 
Under some additional assumptions on $\phi$, we show that we may quasi-isometrically embed Bishop--Jones semigroups of $\Gamma$ into $X$ in such a way that, when restricted to the image of these semigroups, the function $\mathbf{d}_{\phi}$ behaves like a metric, up to a uniform additive error. This is inspired by work of Dey--Kim--Oh in \cite{DKO} and heavily relies on Theorem \ref{GeodCoarseTriangleIneq} due to the same authors. In what follows, for some fixed $\psi \in \theta$, let $\mathcal{T}$ be a Bishop--Jones subsemigroup of $\Gamma$ with respect to the expanding coarse-cocycle $\sigma_{\psi} : \Gamma \times \Lambda_{\theta}(\Gamma) \rightarrow \mathbb{R}$ of Proposition \ref{ConvtoTransv}. Delaying certain definitions for the moment, the main result of this section is the following. 
\begin{thm} [Coarse Triangle Inequality]
\label{CoarseTriangleInequality}
There exists a $\theta$-admissible cone $\mathcal{C} \subset \mathfrak{a}^{+}$ so that: if $\phi \in \mathfrak{a}_{\theta}^{*}$ is positive on $\mathcal{C}_{\theta} \smallsetminus \{0\}$ and $\phi(\kappa(\gamma_{n})) \rightarrow \infty$ for any distinct sequence of elements $\{\gamma_{n}\}$ in $\Gamma$, then there exists a constant $D = D (\phi, \mathcal{T}) > 0$ so that for all $\gamma_{1}, \gamma_{2}, \gamma_{3} \in \mathcal{T}$, we have
\begin{align*}
\mathbf{d}_{\phi}(\gamma_{1} o, \gamma_{3} o) \leq \mathbf{d}_{\phi}(\gamma_{1} o, \gamma_{2} o) + \mathbf{d}_{\phi}(\gamma_{2} o, \gamma_{3} o) + D.
\end{align*}
\end{thm}
Before giving the proof, we recall some relevant notions and results from \cite{DKO}. Set $\mathcal{W}_{\theta} = \{ w \in \mathcal{W} : w(\theta) = \theta \}$.
\begin{defn} [See Section 4 of \cite{DKO}]
Define a closed cone $\mathcal{C}$ in $\mathfrak{a}^{+}$ to be $\theta$-$\textit{admissible}$ if the following three conditions hold:
\begin{itemize}
    \item[(1)] The cone $\mathcal{C}$ is opposition involution invariant: $\iota (\mathcal{C}) = \mathcal{C}$,
    \item[(2)] $\mathcal{W}_{\theta} \cdot \mathcal{C} = \bigcup_{w \in \mathcal{W}_{\theta}} \mathrm{Ad}_{w}\mathcal{C}$ is convex, 
    \item[(3)] $\mathcal{C} \bigcap \big( \bigcup_{\alpha \in \theta} \ker \alpha \big) = \{0\}$.
\end{itemize}
\end{defn}
For a $\theta$-admissible cone $\mathcal{C} \subset \mathfrak{a}^{+}$, we say that an ordered pair of distinct points $(x_{1}, x_{2})$ in $X$ is $\mathcal{C}$-$\textit{regular}$ if for $g_{1}, g_{2} \in G$ with $x_{1} = g_{1}o$ and $x_{2} = g_{2}o$, we have $\kappa(g_{1}^{-1}g_{2}) \in \mathcal{C}$.  
\begin{defn} [Definition 4.2 of \cite{DKO}]
Let $(Z, d_{Z})$ be a metric space and $f : Z \rightarrow X$ a map. For a cone $\mathcal{C} \subset \mathfrak{a}^{+}$ and a constant $B \geq 0$, we say that $f$ is $(\mathcal{C},B)$-$\textit{regular}$ if the pair $(f(z_{1}), f(z_{2}))$ is $\mathcal{C}$-regular for all $z_{1}, z_{2} \in Z$ such that $d_{Z}(z_{1}, z_{2}) \geq B$. We simply say that $f$ is $\mathcal{C}$-$\textit{regular}$ if it is $(\mathcal{C}, B)$-regular for some $B \geq 0$.
\end{defn} 
The importance of this definition comes from the following result of Dey--Kim--Oh in \cite{DKO} which roughly states that the image of a quasi-isometric embedding from a geodesic metric space into the symmetric space $X$ that is $\mathcal{C}$-regular for a $\theta$-admissible cone $\mathcal{C}$ has a metric-like function which satisifes a coarse version of the triangle inequality. This will play a crucial role in the proof of Theorem \ref{CoarseTriangleInequality}; we write $\mathcal{C}_{\theta} = p_{\theta}(\mathcal{C})$.
\begin{thm} [Theorem 4.3 of \cite{DKO}]
\label{GeodCoarseTriangleIneq}
Let $Z$ be a geodesic metric space and $\mathcal{C} \subset \mathfrak{a}^{+}$ a $\theta$-admissible cone. Let $f : Z \rightarrow X$ be a $\mathcal{C}$-regular quasi-isometric embedding. If $\phi \in \mathfrak{a}_{\theta}^{*}$ is positive on $\mathcal{C}_{\theta} \smallsetminus \{0\}$, then there exists a constant $D = D_{\phi} \geq 0$ such that for all $x_{1}, x_{2}, x_{3} \in f(Z)$, 
\begin{align*}
\mathbf{d}_{\phi}(x_{1}, x_{3}) \leq \mathbf{d}_{\phi}(x_{1}, x_{2}) + \mathbf{d}_{\phi}(x_{2}, x_{3}) + D.
\end{align*}
\end{thm}
We can now give the proof of Theorem $\ref{CoarseTriangleInequality}$.
\begin{proof} [Proof of Theorem \ref{CoarseTriangleInequality}]
We view the Bishop--Jones semigroup $\mathcal{T}$ as a geodesic metric space by equipping it with its ``tree metric'' $d_{\mathcal{T}}$, that is, the natural analog in the case of a finitely generated free semigroup of the metric defined on the Cayley graph of a finitely generated free group. By Theorem \ref{GeodCoarseTriangleIneq}, it suffices to construct a $\theta$-admissible cone $\mathcal{C} \subset \mathfrak{a}^{+}$ and a $\mathcal{C}$-$\mathrm{regular}$ quasi-isometric embedding $f : \mathcal{T} \rightarrow X$. Define $f : \mathcal{T} \rightarrow X$ by $f(\gamma) = \gamma o$. We first show that $f$ is a quasi-isometric embedding. That is, we need to show there exist constants $Q > 1$ and $q > 0$ so that 
\begin{align}
\label{QI1}
\frac{1}{Q} d_{\mathcal{T}}(\gamma_{1}, \gamma_{2}) - q \leq d_{X}(\gamma_{1} o, \gamma_{2} o) \leq Q d_{\mathcal{T}}(\gamma_{1}, \gamma_{2}) + q. 
\end{align} Let $S$ denote the finite subset of $\Gamma$ which generates $\mathcal{T}$ as a free semigroup. Given $\gamma_{1}, \gamma_{2} \in \mathcal{T}$, let $\eta$ be the most recent ``ancestor'' of both $\gamma_{1}$ and $\gamma_{2}$. More precisely, $\eta \in S^{m}$ where $m \geq 0$ is the largest integer so that we may write both $\gamma_{1}$ and $\gamma_{2}$ as $\gamma_{1} = \eta \alpha_{1} \cdots \alpha_{k}$ and $\gamma_{2} = \eta \beta_{1} \cdots \beta_{j}$, where the $\alpha_{i}$ and $\beta_{l}$ are in $S$. Notice that $\alpha_{1} \neq \beta_{1}$ (as otherwise $m$ would not be the largest integer with this property). Furthermore, $d_{\mathcal{T}}(\gamma_{1}, \gamma_{2}) = d_{\mathcal{T}}(\eta^{-1} \gamma_{1}, \eta^{-1} \gamma_{2})$ and
\begin{align*}
    d_{X}(\gamma_{1}o, \gamma_{2}o) = ||\kappa (\gamma_{1}^{-1} \gamma_{2})|| = ||\kappa ( (\eta^{-1} \gamma_{1})^{-1} \cdot (\eta^{-1} \gamma_{2})) || = d_{X} (\eta^{-1} \gamma_{1} o, \eta^{-1} \gamma_{2} o).
\end{align*} Hence to prove (\ref{QI1}) in general, it suffices to prove it in the case when $\gamma_{1} = \alpha_{1} \cdots \alpha_{k}$ and $\gamma_{2} = \beta_{1} \cdots \beta_{j}$ are two arbitrary elements of the semigroup with $\alpha_{1} \neq \beta_{1}$. We begin with the easier inequality in (\ref{QI1}), which is the upper bound on $d_{X}(\gamma_{1}o, \gamma_{2}o)$. First, notice that 
\begin{align*}
d_{X}(o, \gamma_{1}o) &= d_{X}(o, \alpha_{1} \cdots \alpha_{k}o) \\
&\leq d_{X}(o, \alpha_{1}o) + d_{X}(\alpha_{1}o, \alpha_{1}\alpha_{2}o) + \cdots + d_{X}(\alpha_{1} \cdots \alpha_{k-1}o, \alpha_{1} \cdots \alpha_{k}o) \\
&= ||\kappa(\alpha_{1})|| + ||\kappa(\alpha_{1}^{-1} \alpha_{1}\alpha_{2})|| + \cdots + ||\kappa(\alpha_{k-1}^{-1} \cdots \alpha_{1}^{-1} \alpha_{1} \cdots \alpha_{k-1} \alpha_{k})|| \\
&= ||\kappa(\alpha_{1})|| + ||\kappa(\alpha_{2})|| + \cdots + ||\kappa(\alpha_{k})|| \\
&\leq \Big( \max_{\zeta \in S} ||\kappa(\zeta)|| \Big) \cdot k \\
&= \Big( \max_{\zeta \in S} ||\kappa(\zeta)|| \Big) \cdot d_{\mathcal{T}}(\mathrm{id}, \gamma_{1}).
\end{align*} Set $L = \max_{\zeta \in S} ||\kappa (\zeta)||$. Likewise, $d_{X}(o, \gamma_{2}o) \leq L d_{\mathcal{T}}(\mathrm{id}, \gamma_{2})$, hence
\begin{align*}
d_{X}(\gamma_{1}o, \gamma_{2}o) \leq d_{X}(\gamma_{1}o,o) + d_{X}(o, \gamma_{2}o) \leq L \big( d_{\mathcal{T}}(\gamma_{1}, \mathrm{id}) + d_{\mathcal{T}}(\mathrm{id}, \gamma_{2}) \big) = L d_{\mathcal{T}}(\gamma_{1}, \gamma_{2}),
\end{align*} which shows that the right-most inequality in (\ref{QI1}) holds with $Q=L$ and $q=0$. It remains to prove the other inequality. First, notice that the magnitude functions $|| \cdot ||_{\sigma_{\phi}}$ and the symmetric space distance $d_{X}$ are related via 
\begin{align*}
||\alpha^{-1} \beta||_{\sigma_{\phi}} = |\phi(\kappa(\alpha^{-1} \beta))| \leq ||\phi||_{\mathrm{op}} \cdot ||\kappa(\alpha^{-1} \beta)|| = ||\phi||_{\mathrm{op}} \cdot d_{X}(\alpha o, \beta o),
\end{align*} hence
\begin{align}
\label{CoarTri1}
d_{X}(\alpha o, \beta o) \geq \frac{1}{||\phi||_{\mathrm{op}}} ||\alpha^{-1} \beta||_{\sigma_{\phi}}
\end{align} for all $\alpha, \beta \in \mathcal{T}$. By part (2) of Theorem \ref{MainThm}, there exist constants $A = A(\sigma_{\phi}) > 1$ and $a = a(\sigma_{\phi}) > 0$ so that 
\begin{align}
\label{CoarTri2}
||\gamma||_{\sigma_{\phi}} \geq \frac{1}{A} |\gamma|_{S} - a = \frac{1}{A} d_{\mathcal{T}}(\mathrm{id}, \gamma) - a
\end{align} for all $\gamma \in \mathcal{T}$ (we call these constants $A$ and $a$ instead of $B_{1}$ and $b_{1}$ as in part (2) of Theorem \ref{MainThm} to avoid causing confusion with the constants $B$ and $b$ of Lemma \ref{ConeEstLem}). As before, let $\gamma_{1} = \alpha_{1} \cdots \alpha_{k}$ and $\gamma_{2} = \beta_{1} \cdots \beta_{j}$ be two arbitrary elements of $\mathcal{T}$ with $\alpha_{1} \neq \beta_{1}$. Recall that we denote by $d$ the metric on the compact space $\Gamma \sqcup \Lambda(\Gamma)$, and that this metric is \emph{not} $\Gamma$-invariant. By Lemma \ref{DefiniteDistance}, there is a constant $r > 0$ (independent of $\gamma_{1}$ and $\gamma_{2}$) so that $d (\gamma_{1}, \gamma_{2}) \geq r$. Let $C = C(r)$ be the constant provided by part (3) of Proposition \ref{CocycleProperties} when taking $\epsilon = r$. Then $d(\gamma_{1}, \gamma_{2}) \geq r$ implies that 
\begin{align}
\label{CoarTri3}
||\gamma_{1}^{-1} \gamma_{2}||_{\sigma_{\phi}} \geq ||\gamma_{1}^{-1}||_{\sigma_{\phi}} + ||\gamma_{2}||_{\sigma_{\phi}} - C.
\end{align} Applying Proposition \ref{GPSProperties} and part (5) of Theorem \ref{MainThm} to the GPS system $(\sigma_{\phi}, \bar{\sigma}_{\phi}, \phi \circ G_{\theta})$, we obtain a constants $D > 1$ and $M > 0$ so that 
\begin{align*}
||\gamma^{-1}||_{\sigma_{\phi}} \geq \frac{1}{D} ||\gamma||_{\sigma_{\phi}} - M
\end{align*} for all $\gamma \in \mathcal{T}$. Thus, (\ref{CoarTri3}) becomes
\begin{align}
\label{CoarTri4}
||\gamma_{1}^{-1}\gamma_{2}||_{\sigma_{\phi}} \geq \frac{1}{D} ||\gamma_{1}||_{\sigma_{\phi}} + ||\gamma_{2}||_{\sigma_{\phi}} - (M + C).
\end{align} Combining (\ref{CoarTri1}), (\ref{CoarTri2}), and (\ref{CoarTri4}), we obtain
\begin{align*}
d_{X}(\gamma_{1}o, \gamma_{2}o) &\geq \frac{1}{||\phi||_{\mathrm{op}}} ||\gamma_{1}^{-1} \gamma_{2}||_{\sigma_{\phi}} \\
&\geq \frac{1}{||\phi||_{\mathrm{op}}} \bigg( \frac{1}{D} ||\gamma_{1}||_{\sigma_{\phi}} + ||\gamma_{2}||_{\sigma_{\phi}} - (M+C) \bigg) \\
&\geq \frac{1}{||\phi||_{\mathrm{op}}} \bigg( \frac{1}{D} \Big( \frac{1}{A} d_{\mathcal{T}}(\mathrm{id}, \gamma_{1}) - a \Big) + \Big(\frac{1}{A} d_{\mathcal{T}}(\mathrm{id}, \gamma_{2}) - a \Big)  - (M+C) \bigg) \\
&\geq \frac{1}{||\phi||_{\mathrm{op}}} \bigg( \frac{1}{AD} \big( d_{\mathcal{T}}(\gamma_{1}, \mathrm{id}) + d_{\mathcal{T}}(\mathrm{id}, \gamma_{2}) \big) - \Big(\frac{a}{D} + a + M + C \Big) \bigg) \\
&= \frac{1}{AD \cdot ||\phi||_{\mathrm{op}}} \cdot d_{\mathcal{T}} (\gamma_{1}, \gamma_{2}) - \frac{\frac{a}{D} + a + M + C}{||\phi||_{\mathrm{op}}}.
\end{align*} Thus (\ref{QI1}) holds with 
\begin{align*}
Q = \max \{L, AD \cdot ||\phi||_{\mathrm{op}} \} \  \mathrm{and} \ q = \frac{\frac{a}{D} + a + M + C}{||\phi||_{\mathrm{op}}},
\end{align*} which shows that the map $f : \mathcal{T} \rightarrow X$ is indeed a quasi-isometric embedding.
\par 
In order to finish the proof, we need the following technical lemma.
\begin{lem}
\label{ConeEstLem}
There exist $B,b > 0$ so that, for any two elements $\gamma_{1}, \gamma_{2} \in \mathcal{T}$ with $d_{\mathcal{T}}(\gamma_{1}, \gamma_{2}) \geq B$, we have
\begin{align}
\label{ConeEst1}
\min_{\alpha \in \theta} ||\gamma_{1}^{-1} \gamma_{2}||_{\sigma_{\alpha}} = \min_{\alpha \in \theta} \alpha (\kappa (\gamma_{1}^{-1} \gamma_{2})) \geq b ||\kappa (\gamma_{1}^{-1} \gamma_{2})||.
\end{align}
\end{lem}
\begin{proof} [Proof of Lemma \ref{ConeEstLem}]
As in the verification of (\ref{QI1}), it suffices to prove the lemma in the case when $\gamma_{1} = \alpha_{1} \cdots \alpha_{k}$ and $\gamma_{2} = \beta_{1} \cdots \beta_{j}$ are two arbitrary elements of the semigroup with $\alpha_{1} \neq \beta_{1}$. By part (5) of Theorem \ref{MainThm}, there exist constants $P > 1$ and $p > 0$ so that for all $\gamma_{1}, \gamma_{2} \in \mathcal{T}$, we have
\begin{align*}
\min_{\alpha \in \theta} ||\gamma_{1}^{-1} \gamma_{2}||_{\sigma_{\alpha}} \geq \frac{1}{P} ||\gamma_{1}^{-1} \gamma_{2}||_{\sigma_{\phi}} - p.
\end{align*} Hence it suffices to show that there exist $B,b > 0$ so that if $d_{\mathcal{T}}(\gamma_{1}, \gamma_{2}) \geq B$, then
\begin{align}
\label{ConeEst2}
||\gamma_{1}^{-1} \gamma_{2}||_{\sigma_{\phi}} = \phi(\kappa(\gamma_{1}^{-1} \gamma_{2})) \geq bP ||\kappa(\gamma_{1}^{-1} \gamma_{2})|| + Pp.
\end{align} Let $A,a,M,C,D$ be the constants of inequalities (\ref{CoarTri2}), (\ref{CoarTri3}), and (\ref{CoarTri4}), and let $Q,q$ be the quasi-isometric embedding constants of (\ref{QI1}). Note that $\gamma_{1}$ and $\gamma_{2}$ satisfy the assumptions of Lemma \ref{DefiniteDistance}, so we may apply each of the aforementioned inequalities. By (\ref{QI1}), (\ref{CoarTri2}), and the fact that $\gamma_{1}$ and $\gamma_{2}$ only have the identity as a common ancestor, we have
\begin{align}
\label{ConeEst3}
\begin{split}
||\gamma_{1}||_{\sigma_{\phi}} + ||\gamma_{2}||_{\sigma_{\phi}} + \Big( 2a + \frac{q}{AQ} \Big) &\geq \frac{1}{A} d_{\mathcal{T}} (\gamma_{1}, \mathrm{id}) - a + \frac{1}{A} d_{\mathcal{T}} (\mathrm{id}, \gamma_{2}) - a + 2a + \frac{q}{AQ} \\
&= \frac{1}{A} d_{\mathcal{T}}(\gamma_{1}, \gamma_{2}) + \frac{q}{AQ} \\
&\geq \frac{1}{A} \bigg( \frac{1}{Q} d_{X}(\gamma_{1}o, \gamma_{2}o) - \frac{q}{Q} \bigg) + \frac{q}{AQ} \\
&= \frac{1}{AQ} ||\kappa (\gamma_{1}^{-1} \gamma_{2})||.
\end{split}
\end{align} Fix a constant $H > D > 1$. We claim that there exists a constant $B \geq 1$ so that if $d_{\mathcal{T}}(\gamma_{1}, \gamma_{2}) \geq B$, then either
\begin{align}
\label{option1}
||\gamma_{1}||_{\sigma_{\phi}} \geq \frac{2a + \frac{q}{AQ} + M + C + HPp}{\frac{1}{D} - \frac{1}{H}},
\end{align} or
\begin{align}
\label{option2}
||\gamma_{2}||_{\sigma_{\phi}} \geq \frac{2a + \frac{q}{AQ} + M + C + HPp}{1 - \frac{1}{H}}. 
\end{align} Indeed, if not, then for every integer $n \geq 1$, there exist elements $\alpha_{n}$ and $\beta_{n}$ of $\mathcal{T}$ having the identity as their only common ancestor, satisfying $d_{\mathcal{T}}(\alpha_{n}, \beta_{n}) \geq n$, but such that the sequence $\{||\alpha_{n}||_{\sigma_{\phi}}\}$ is bounded above by the right-hand side of (\ref{option1}), while the sequence  $\{||\beta_{n}||_{\sigma_{\phi}}\}$ is bounded above by the right-hand side of (\ref{option2}). Since $d_{\mathcal{T}}(\alpha_{n}, \beta_{n}) \rightarrow \infty$ as $n \rightarrow \infty$, there either exists a subsequence of $\{\alpha_{n}\}$ or a subsequence of $\{\beta_{n}\}$ which is escaping. But any such subsequence must have $\sigma_{\phi}$-magnitudes tending to infinity, which is a contradiction. We conclude that there must exist some $B \geq 1$ so that $d_{\mathcal{T}}(\gamma_{1}, \gamma_{2}) \geq B$ implies that either (\ref{option1}) or (\ref{option2}) holds. \par Now if $d_{\mathcal{T}}(\gamma_{1}, \gamma_{2}) \geq B$, then combining (\ref{CoarTri4}), (\ref{ConeEst3}), (\ref{option1}), and (\ref{option2}), a direct computation shows that
\begin{align*}
||\gamma_{1}^{-1} \gamma_{2}||_{\sigma_{\phi}} &\geq \frac{1}{D} ||\gamma_{1}||_{\sigma_{\phi}} + ||\gamma_{2}||_{\sigma_{\phi}} - (M+C) \\
&> \frac{1}{H} \bigg( ||\gamma_{1}||_{\sigma_{\phi}} + ||\gamma_{2}||_{\sigma_{\phi}} + 2a + \frac{q}{AQ} + HPp \bigg) \\
&> \frac{1}{AHQ} ||\kappa (\gamma_{1}^{-1} \gamma_{2})|| + Pp.
\end{align*} Thus (\ref{ConeEst2}) holds with $B$ as above and $b = \frac{1}{AHQP}$, concluding the proof of the lemma.
\end{proof}
We now construct a $\theta$-admissible cone $\mathcal{C} \subset \mathfrak{a}^{+}$ so that the quasi-isometric embedding $f : \mathcal{T} \rightarrow X$ is $\mathcal{C}$-regular. Let $B,b > 0$ be the constants of Lemma \ref{ConeEstLem}. Define
\begin{align*}
U_{0} := \{ \kappa (\gamma_{1}^{-1} \gamma_{2}) : \gamma_{1}, \gamma_{2} \in \mathcal{T}, d_{\mathcal{T}}(\gamma_{1}, \gamma_{2}) \geq B \} \subset \mathfrak{a}^{+}.
\end{align*} Let $U \subset \mathfrak{a}_{\theta}^{+}$ denote the convex hull of $p_{\theta}(U_{0})$ and consider the asymptotic cone of $U$, that is
\begin{align*}
    \mathcal{C}_{0} := \overline{\bigcup_{u \in U} \mathbb{R}^{+}u}.
\end{align*} Since $U_{0}$ is invariant under the opposition involution, so is $\mathcal{C}_{0}$. We claim that $\mathcal{C}_{0} - \{0\} \subset \mathrm{Int}(\mathfrak{a}_{\theta}^{+})$. Let $u \in U$. Then $u$ is of the form $u = \sum_{i=1}^{k} t_{i}u_{i}$ where $t_{i} \geq 0$ for all $i$, $\sum_{i=1}^{k} t_{i} = 1$, and $u_{i} \in p_{\theta}(U_{0})$ for all $i$. By Lemma \ref{ConeEstLem}, for any $\alpha \in \theta$ we have
\begin{align}
\label{ConeEst6}
\alpha (u) = \sum_{i=1}^{k} t_{i} \alpha (u_{i}) \geq b \sum_{i=1}^{k} t_{i} ||u_{i}|| \geq b \Big| \Big| \sum_{i=1}^{k} t_{i} u_{i} \Big| \Big| = b ||u||.
\end{align} Now let $x \in \mathcal{C}_{0} \smallsetminus \{0\}$ be arbitrary. Then there exist sequences $\{\lambda_{n}\} \subset \mathbb{R}^{+}$ and $\{u_{n}\} \subset U$ so that $\lambda_{n} u_{n} \to x$ as $n \to \infty$. For any $\alpha \in \theta$, (\ref{ConeEst6}) implies that
\begin{align}
\label{ConeEst7}
\alpha(x) = \lim_{n \to \infty} \alpha (\lambda_{n} u_{n}) \geq \lim_{n \to \infty} \lambda_{n} \cdot b ||u_{n}|| = b \lim_{n \to \infty} ||\lambda_{n} u_{n}|| = b ||x|| > 0.
\end{align} So we indeed have $\mathcal{C}_{0} - \{0\} \subset \mathrm{Int}(\mathfrak{a}_{\theta}^{+})$. We now follow the argument of the proof of Theorem 4.1 in \cite{DKO} to construct the desired $\theta$-admissible cone. Define
\begin{align*}
    \mathcal{C} := p_{\theta}^{-1}(\mathcal{C}_{0}) \cap \mathfrak{a}^{+}.
\end{align*} By construction, $U_{0} \subset \mathcal{C}$. Since the projection $p_{\theta} : \mathfrak{a} \rightarrow \mathfrak{a}_{\theta}$ is $\mathcal{W}_{\theta}$-equivariant, we have
\begin{align*}
    \mathcal{W}_{\theta} \cdot \mathcal{C} = \mathcal{W}_{\theta} \cdot (p_{\theta}^{-1}(\mathcal{C}_{0}) \cap \mathfrak{a}^{+}) = p_{\theta}^{-1}(\mathcal{C}_{0}) \cap (\mathcal{W}_{\theta} \cdot \mathfrak{a}^{+}).
\end{align*} As $\mathcal{W}_{\theta} \cdot \mathfrak{a}^{+}$ and $p_{\theta}^{-1}(\mathcal{C}_{0})$ are both convex cones, their intersection $\mathcal{W}_{\theta} \cdot \mathcal{C}$ is also convex. By (\ref{ConeEst7}), we can remove some sufficiently small conical neighborhoods of $\bigcup_{\alpha \in \theta} \ker \alpha$ from $\mathcal{C}$ so that we may assume that 
\begin{align*}
    \mathcal{C} \cap \bigcup_{\alpha \in \theta} \ker \alpha = \{0\},
\end{align*} while still having $U_{0} \subset \mathcal{C}$ and still preserving the convexity of $\mathcal{W}_{\theta} \cdot \mathcal{C}$. Thus $\mathcal{C}$ is $\theta$-admissible and, by the definition of $U_{0}$, the quasi-isometric embedding $f : \mathcal{T} \rightarrow X$ is $(\mathcal{C}, B)$-regular. This concludes the proof.
\end{proof}

\end{document}